\documentclass[10pt,a4paper]{amsart}

\newcommand*\patchAmsMathEnvironmentForLineno[1]{%
  \expandafter\let\csname old#1\expandafter\endcsname\csname #1\endcsname
  \expandafter\let\csname oldend#1\expandafter\endcsname\csname end#1\endcsname
   \renewenvironment{#1}%
     {\linenomath\csname old#1\endcsname}%
     {\csname oldend#1\endcsname\endlinenomath}}%
\newcommand*\patchBothAmsMathEnvironmentsForLineno[1]{%
  \patchAmsMathEnvironmentForLineno{#1}%
  \patchAmsMathEnvironmentForLineno{#1*}}%
\AtBeginDocument{%
\patchBothAmsMathEnvironmentsForLineno{equation}%
\patchBothAmsMathEnvironmentsForLineno{align}%
\patchBothAmsMathEnvironmentsForLineno{flalign}%
  \patchBothAmsMathEnvironmentsForLineno{alignat}%
\patchBothAmsMathEnvironmentsForLineno{gather}%
\patchBothAmsMathEnvironmentsForLineno{multline}%
}

\usepackage{amsthm, amsfonts, amssymb, amsmath, latexsym, enumerate}
\usepackage[all]{xy} \usepackage[latin1]{inputenc}
\usepackage{array}
\usepackage{hyperref}
\usepackage{cite}

\usepackage{mathrsfs}
\newcommand{\mathbold}{\mathbf}

\usepackage{color}

\newtheorem{thm}{Theorem}[section] \newtheorem{lem}[thm]{Lemma}
\newtheorem{corol}[thm]{Corollary} \newtheorem{prop}[thm]{Proposition}
\newtheorem{claim}[thm]{Claim}
\newtheorem*{thm*}{Theorem}
\newtheorem{step}{Step}

\theoremstyle{definition} \newtheorem{rmk}[thm]{Remark}
 \newtheorem{dfn}[thm]{Definition}

\newtheorem{case}{Case}
\newtheorem*{ack}{Acknowledgements}

\newcommand{\OO}{\mathscr{O}}

\newcommand{\sF}{\mathscr{F}}

\newcommand{\EExt}{\mathscr{E}xt}
\newcommand{\HHom}{\mathscr{H}om}
\newcommand{\TTor}{\mathscr{T}or}

\newcommand{\sE}{\mathscr{E}}
\newcommand{\EE}{\sE}
\newcommand{\sH}{\mathscr{H}}
\newcommand{\sK}{\mathscr{K}}

\newcommand{\sS}{\mathscr{S}}
\newcommand{\sP}{\mathscr P}
\newcommand{\sI}{\mathscr{I}}

\newcommand{\cH}{\mathcal{H}}
\newcommand{\cI}{\mathcal{I}}

\newcommand{\cS}{\mathcal{S}}

\newcommand{\Mo}{{\sf M}}

\newcommand{\Spl}{{\sf Spl}}
\newcommand{\No}{{\sf N}}

\DeclareMathOperator{\Pf}{Pf}

\DeclareMathOperator{\rk}{rk}

\DeclareMathOperator{\Ext}{Ext} 

\DeclareMathOperator{\Hom}{Hom} 
\DeclareMathOperator{\im}{Im} 
\DeclareMathOperator{\HH}{H} \DeclareMathOperator{\hh}{h}
\DeclareMathOperator{\ext}{ext}

\DeclareMathOperator{\Pic}{Pic}

\DeclareMathOperator{\depth}{depth}

\newcommand{\Z}{\mathbb Z} \newcommand{\C}{\mathbb C}
 
 \newcommand{\p}{\mathbb P}

\newcommand{\RR}{\mathbold{R}}

\DeclareMathOperator{\ts}{\otimes}
\newcommand{\rr}{\rightarrow}

\newcommand{\mono}{\hookrightarrow}

\newcommand{\xr}{\xrightarrow}
\newcommand{\f}{F}

\newcommand{\FD}{F^D} 

\SelectTips{cm}{} \newdir{ >}{{}*!/-5pt/@{>}}
\numberwithin{equation}{section}

\usepackage{lineno}

\allowdisplaybreaks[1]
\begin{document}



\title[ACM bundles on prime Fano threefolds]{Moduli spaces of
  rank 2 ACM bundles \\ on prime Fano threefolds}
\author{Maria Chiara Brambilla}
\email{{\tt brambilla@dipmat.univpm.it}}
\address{Dipartimento di Scienze Matematiche,
  Università Politecnica delle Marche\\
  Via Brecce Bianche, I-60131 Ancona - Italia}
\urladdr{\tt{\url{http://www.dipmat.univpm.it/~brambilla}}}

\author{Daniele Faenzi}
\email{{\tt daniele.faenzi@univ-pau.fr}}
\address{Université de Pau et des Pays de l'Adour \\
  Avenue de l'Universit\'e - BP 576 - 64012 PAU Cedex - France}
\urladdr{\tt{\url{http://www.univ-pau.fr/~faenzi/}}}

\keywords{Prime Fano threefolds.
  Moduli space of vector bundles.
  Bundles with no intermediate cohomology.
  Arithmetically Cohen-Macaulay (ACM) sheaves.}

\subjclass[2000]{Primary 13C14, 14J60. Secondary 14F05, 14D20.}

\thanks{Both authors were partially supported by Italian MIUR funds.
The second named author was partially supported by GRIFGA and ANR
contract ANR-09-JCJC-0097-01}

 \begin{abstract}
   Given a smooth non-hyperelliptic prime Fano threefold $X$, we prove the
   existence of all rank $2$ ACM vector bundles on $X$ by deformation
   of semistable sheaves. We show that these bundles move in
   generically smooth components of the corresponding moduli space.
   
   We give applications to pfaffian
   representations of quartic threefolds in $\p^4$ and
   cubic hypersurfaces of a smooth quadric of $\p^5$.
 \end{abstract}

\maketitle

\section{Introduction}
A vector bundle $F$ on a smooth polarized variety $(X,H_X)$ has no
intermediate cohomology if: $\HH^k(X,F \ts \OO_X(t\,H_X))=0$ for all
$t\in \Z$ and $0 < k < \dim(X)$.
These bundles are also called {\it arithmetically Cohen-Macaulay
(ACM)}, because they correspond to maximal Cohen-Macaulay modules
over the coordinate ring of $X$.
It is known that an ACM bundle must be a direct sum of line bundles if
$X=\p^n$ (see \cite{horrocks:punctured}), or a direct sum of line
bundles and (twisted) spinor bundles if $X$ is a smooth quadric
hypersurface in $\p^n$ (we refer to \cite{knorrer:aCM},
\cite{ottaviani:horrocks}).
On the other hand, there exists a complete classification of varieties
admitting, up to twist, a finite number of isomorphism classes of
indecomposable ACM bundle see \cite{eisenbud-herzog:CM},
\cite{buchweitz-greuel-schreyer}.
Only five cases exist besides rational normal curves, projective spaces
and quadrics.

For varieties which are not in this list,
the problem of classifying ACM bundles has been taken up only in some special cases.
For instance, on general
hypersurfaces in $\p^n$, of dimension at least $3$, a full classification of ACM
bundles of rank 2 is
available, see \cite{kumar-rao-ravindra:hypersurfaces}, \cite{kumar-rao-ravindra:three-dimensional},
\cite{chiantini-madonna:international},
\cite{chiantini-madonna:matematiche}, \cite{madonna:quartic}.
For dimension $2$ and rank $2$, a partial classification can be found in
\cite{dani:cubic:ja}, \cite{dani-chiantini:to-appear}, \cite{beauville:determinantal}
\cite{daniele-chiantini-bis}.
For higher rank, some results are given in
 \cite{casanellas-hartshorne:cubic}, \cite{arrondo-madonna}.

The case of smooth Fano threefolds $X$ with Picard number 1 has also
been studied.
In this case one has $\Pic(X) \cong \langle H_X \rangle$, with $H_X$
ample, and the canonical divisor class $K_X$ satisfies $K_X =-i_{X} \, H_X$, where the {\it index} $i_X$ satisfies
$1 \leq i_{X} \leq 4$.
Recall that $i_{X}=4$ implies $X \cong \p^3$
and $i_{X}=3$ implies that $X$ is isomorphic to a smooth quadric.
Thus, the class of ACM bundles is completely understood in these two cases.

In contrast to this, the cases $i_{X}=2,1$ are highly nontrivial.
First of all, there are several deformation classes of these
varieties, see \cite{iskovskih:I}, \cite{iskovskih:II},
\cite{fano-encyclo}.
A different approach to the classification of these varieties was
proposed by Mukai, see for instance \cite{mukai:curves-K3}, \cite{mukai:biregular},
\cite{mukai:curves-symmetric-I}.

In second place, it is still unclear how to characterize the
invariants of ACM bundles: in fact the investigation has been
deeply carried out only in the case of rank $2$. For $i_{X}=2$,
the classification was completed in \cite{arrondo-costa}. For
$i_{X}=1$, a result of Madonna (see \cite{madonna:fano-cy})
implies that if a rank $2$ ACM bundle $F$ is defined on
$X$, then its second Chern class $c_2$ must take values in $\{1,
\ldots, g+3 \}$ if $c_1(F)=1$,
or in $\{2,4 \}$ if $c_1(F) = 0$. Partial existence results are given in \cite{madonna:quartic},
\cite{enrique-dani:v22}.

In third place, the set of ACM rank $2$ bundles can have positive
dimension. A natural point of view is to study them in terms of the
moduli space $\Mo_{X}(2,c_{1},c_{2})$ of (Gieseker)-semistable
rank $2$ sheaves $F$ with $c_{1}(F)=c_{1}$, $c_{2}(F)=c_{2}$,
$c_{3}(F)=0$.
For $i_{X}=2$, such moduli space has been mostly studied when $X$ is a
smooth cubic threefold, see for instance
\cite{markushevich-tikhomirov}, \cite{iliev-markushevich:cubic},
\cite{druel:cubic-3-fold}, see also \cite{beauville:cubic} for a
survey.

If the index $i_{X}$ equals $1$, the threefold $X$ is said to be
{\it prime}, and one defines the {\it genus} of $X$ as $g = 1+
H_X^3/2$.
The genus satisfies $2\leq g \leq 12$, $g\neq 11$, and
there are $10$ deformation classes of prime Fano threefolds.
In this case, some of the relevant moduli spaces $\Mo_{X}(2,1,c_{2})$
are studied in \cite{iliev-markushevich:quartic} (for $g=3$),
\cite{iliev-markushevich:genus-7},
\cite{iliev-markushevich:sing-theta:asian},
\cite{brambilla-faenzi:genus-7} (for $g=7$),
\cite{iliev-manivel:genus-8}, \cite{iliev-markushevich:cubic} (for
$g=8$), \cite{iliev-ranestad} (for $g=9$), \cite{enrique-dani:v22}
(for $g=12$).

ACM bundles of rank $2$ also appeared in the framework of
determinantal hypersurfaces, indeed any such bundle provides a pfaffian
representation of the equation of the hypersurface, see
e.g. \cite{beauville:determinantal}, \cite{iliev-markushevich:quartic}.

\vspace{0.2cm}

The goal of our paper is to provide the classification of rank $2$ ACM
bundles $F$ on a smooth prime Fano threefold $X$, i.e. in the case $i_X=1$.
Note that we can assume $c_1(F)\in\{0,1\}$.
Combining our existence theorems (namely Theorem \ref{topolino} and
Theorem \ref{thm:caso4}) with the results of Madonna and others
mentioned above, we obtain the following classification.


\begin{thm*}
Let $X$ be a smooth prime Fano threefold of genus $g$,
with $-K_X$ very ample. 
Then an ACM vector bundle $F$ of rank 2 has the following Chern classes:
\begin{enumerate}[i)]
\item \label{olala}
if $c_1(F) = 1$  then $c_{2}(F)=1$ or $\frac{g}{2}+1 \le c_{2}(F) \le  g + 3$. 
\item \label{trullalla}
if $c_1(F)=0$ then $c_2(F)=2,4$.
\end{enumerate}

If $c_1(F)=1$ and $c_2(F)\geq \frac{g}{2}+2$, we assume, in addition, 
that $X$ contains a line $L$ with normal bundle $\OO_L \oplus \OO_L(-1)$.
Then there exists an ACM vector bundle $F$ for any case listed above.
\end{thm*}

Note that the assumption that $-K_X$ is very ample (the threefold $X$
is then called {\it non-hyperelliptic}) excludes two families of prime
Fano threefolds, one with $g=2$, the other with $g=3$.
These two cases will be discussed in a forthcoming paper.

The proof is based on deformations of sheaves which are not locally
free (hence neither ACM) such as extensions of ideal sheaves.
The idea is to work recursively starting by some well-behaved bundles
with {\it minimal} $c_{2}$.
In order for the induction to work in case \eqref{olala}, the only hypothesis we need on 
the threefold $X$ is to contain a line $L$ with
normal bundle $\OO_L \oplus \OO_L(-1)$ (in this case we will say that
$X$ is {\it ordinary}).
This is always verified if $g\ge 9$ unless $X$ is
the Mukai-Umemura threefold of genus $12$.
This condition is
verified by a general prime Fano threefold of any genus, see
Section \ref{prime-fano} for more details.

The paper is organized as follows. In the next section we give
some preliminary notions. Section \ref{sec:odd} is devoted to the
proof 
of the main theorem in the case $c_1(F)=1$,
while Section \ref{sec:even} concerns the case $c_1(F)=0
$. 
We conclude the paper with Section \ref{apps}, giving applications to pfaffian
representations of quartic threefolds in $\p^{4}$ and cubic
hypersurfaces of a smooth quadric in $\p^{5}$.

\begin{ack}
We are grateful to Atanas Iliev and Giorgio Ottaviani
for several useful remarks. We thank the referee for
suggesting several improvements to the paper.
\end{ack}

\section{Preliminaries}

Given a smooth complex projective $n$-dimensional polarized variety $(X,H_X)$
and a sheaf $F$ on $X$, we write $F(t)$ for $F\ts \OO_X(t H_X)$.
Given a subscheme $Z$ of $X$, we write $F_Z$ for $F\ts \OO_Z$ and
we denote by $\cI_{Z,X}$ the ideal sheaf of $Z$ in $X$, and by
$N_{Z,X}$ its normal sheaf. We will frequently drop
the second subscript.

Given a pair of sheaves $(F,E)$ on $X$, we will write $\ext_X^k(F,E)$
for the dimension of the 
group $\Ext_X^k(F,E)$, and
similarly $\hh^k(X,F) = \dim \HH^k(X,F)$.
The Euler characteristic of a pair of sheaves $(F,E)$
is defined as $\chi(F,E)=\sum_k (-1)^k \ext_X^k(F,E)$ and
$\chi(F)$ is defined as $\chi(\OO_X,F)$.
We denote by $p(F,t)$ the Hilbert polynomial $\chi(F(t))$ of the sheaf
$F$.
We write $e_{E,F}$ for the natural evaluation map
\[
e_{E,F}:\Hom_X(E,F)\ts E \to F.
\]
\subsection{ACM sheaves}\label{sezioneACM}

Let $(X,H_X)$ be a $n$-dimensional polarized variety, and assume $H_X$ very ample, so
we have $X \subset \p^m$. We denote by $I_X$ the saturated ideal of $X$ in
$\p^m$, and by $R(X)$ the coordinate ring of $X$. Given a sheaf $F$ on
$X$, we define the following $R(X)$-modules:
\[
\HH_*^k(X,F) = \bigoplus_{t \in \Z} \HH^k(X,F(t)), \qquad \mbox{for each $k=0,\ldots,n$.}
\]

The variety $X$ is
said to be arithmetically Cohen-Macaulay (ACM) if $R(X)$ is a
Cohen-Macaulay ring. This is equivalent to
$\HH_*^1(\p^m,\cI_{X,\p^m})=0$ and $\HH_*^k(\p^m,\OO_X)=0$ for $0<k<n$.
A sheaf $F$ on $X$ is called locally Cohen-Macaulay if for
any point $x\in X$ we have $\depth(F_x)=\dim(X)$.

\begin{dfn}
A sheaf $F$ on an $n$-dimensional ACM variety $X$ is called {\em ACM} if $F$ is locally Cohen-Macaulay and it has no
intermediate cohomology:
\begin{equation}
  \label{eq:ACM}
  \HH_*^{k}(X,F) = 0 \quad \mbox{for any $0<k<n$ }.
\end{equation}
\end{dfn}

By \cite[Proposition 2.1]{hartshorne-casanellas:biliaison}, there is a
one-to-one correspondence between ACM sheaves on $X$ and graded
maximal Cohen-Macaulay modules on $R(X)$, given by $F\mapsto
\HH_*^{0}(X,F)$.
If $X$ is smooth, any ACM sheaf is locally free 
(see e.g. \cite[Lemma 3.2]{abe-yoshinaga}), so $F$ being ACM is
equivalent to condition \eqref{eq:ACM}.

As already mentioned, on a smooth quadric hypersurface of $\p^m$, with
$m\geq 4$, there
exist ACM bundles of rank greater than $1$, called {\it spinor bundles}.
We recall here some facts and notation on these bundles in case they
have rank $2$, for more details we refer to \cite{ottaviani:spinor}, \cite{knorrer:aCM}.
If $Q_3\subset\p^4$ is a smooth quadric, then there exists one spinor
bundle $\cS$ of rank $2$ on $Q_3$.
It is $\mu$-stable (see Section \ref{vachement} below), globally
generated, with first Chern class equal to the 
hyperplane class $H_{Q_3}$ of $Q_3$ and with
$c_{2}(\cS)=[L]$, where $L$ is a line contained in $Q_3$.
Moreover, we have the natural exact sequence on $Q_3$:
  \begin{equation}
    \label{spinorazzi-unosolo}
    0 \to \cS(-1) \to \OO^4_{Q_3} \xr{e_{\OO,\cS}} \cS \to 0.    
  \end{equation}

On the other hand, if $Q_4\subset\p^5$ is a smooth quadric, then there exist two
non-isomorphic spinor bundles of rank $2$ defined over $Q_4$. 
We denote them by $\cS_1$ and $\cS_2$.
They are both $\mu$-stable (see Section \ref{vachement} below), globally generated and satisfy $c_1(\cS_i)=H_{Q_4}$ 
and $c_{2}(\cS_{i})=\Lambda_{i}$, where $\Lambda_{1}$ and
$\Lambda_{2}$ are the classes of two disjoint projective planes
contained in $Q_4$.
These planes are parametrized by global sections of the bundles $\cS_i$.
These classes generate the cohomology group
$\HH^{2,2}(Q_4)$, 
and one has the relations: 
$H_{Q_4}^{2}=\Lambda_{1} + \Lambda_{2}$ and $\Lambda_{i}^{2}=1$.
Moreover, we have the natural exact sequences on $Q_4$:
  \begin{equation}
    \label{spinorazzi}
    0 \to \cS_i(-1) \to \OO^4_{Q_4} \xr{e_{\OO,\cS_{i+1}}} \cS_{i+1} \to 0,    
  \end{equation}
where we take the indices mod $2$.

\subsection{Stability and moduli spaces} \label{vachement}
Let us now recall a few well-known facts about semistable sheaves on projective varieties.
We refer to the book \cite{huybrechts-lehn:moduli} for a more detailed account of these notions.

Let $(X,H_X)$ be a smooth complex projective $n$-dimensional polarized variety.
We recall that a torsion-free coherent sheaf $F$ on $X$ is (Gieseker) {\it semistable} if for
any coherent subsheaf $E$, with $0<\rk(E)<\rk(F)$,
one has $p(E,t)/\rk(E) \leq p(F,t)/\rk(F)$ for $t\gg 0$.
The sheaf $F$ is called {\it stable} if the inequality above is always strict.

If $X$ has Picard number 1, we can think the first
Chern class $c_1(F)$ of a sheaf $F$ on $X$ as an integer.
Then the {\it slope} of a sheaf $F$ of positive rank is defined as
$\mu(F) = c_1(F)/\rk(F)$.
We say that $F$ is {\em normalized} if $-1 < \mu(F) \leq 0$.
We recall that a torsion-free coherent sheaf $F$ is  {\it $\mu$-semistable} if for
any coherent subsheaf $E$, with $0<\rk(E)<\rk(F)$,
one has $\mu(E) \leq \mu(F)$.
The sheaf $F$ is called {\it $\mu$-stable} if the above inequality is always strict.
The two notions are related by the following implications:
$$\mbox{$\mu$-stable} \quad\Rightarrow\quad \mbox{stable}\quad \Rightarrow\quad \mbox{semistable}
\quad\Rightarrow \quad\mbox{$\mu$-semistable}$$
Notice that a rank $2$ sheaf $F$ with odd $c_1(F)$ is $\mu$-stable
as soon as it is $\mu$-semistable.

Recall that by Maruyama's theorem, see
\cite{maruyama:boundedness-small}, if $\dim(X)=n\geq 2$ and $F$ is a
$\mu$-semistable sheaf of rank $r<n$, then
its restriction to a general hyperplane section of $X$ is still $\mu$-semistable.

Let us introduce some notation concerning moduli spaces.
We denote by $\Mo_X(r,c_1,\ldots,c_n)$ the moduli space of
$S$-equivalence classes of rank $r$
torsion-free semistable sheaves on $X$ with Chern classes $c_1,\ldots,c_n$,
where $c_k$ lies in $\HH^{k,k}(X)$.
For brevity, sometimes we will write $F$ instead of the class $[F]$.

The Chern class $c_k$ will be denoted by an integer as soon as
$\HH^{k,k}(X)$ has dimension $1$. We will drop the last values of the
classes $c_k$ when they are zero.

We will denote by $\Spl_X$ the coarse moduli space of simple sheaves on $X$.
As proved in \cite{altman-kleiman:compactifying}, it is an algebraic space.

We denote by $\sH^g_d(X)$ the union of components of the Hilbert scheme of closed
$Z$ subschemes of $X$ with Hilbert polynomial $p(\OO_Z,t)=dt+1-g$, containing integral curves of degree $d$
and arithmetic genus $g$.

\subsection{Prime Fano threefolds} \label{prime-fano}

Let now $X$ be a smooth projective variety of dimension $3$.
Recall that $X$ is called {\it Fano} if its anticanonical divisor
class $-K_X$ is ample.
A Fano threefold is called {\it non-hyperelliptic} if $-K_X$ is very ample.

We say that $X$ is {\it prime} if the Picard group
is generated by the canonical divisor class $K_X$.
If $X$ is a prime Fano threefold we have
$\Pic(X)\cong \Z \cong \langle H_X \rangle$,
where $H_X = -K_X$ is ample.
One defines the {\it genus} of a prime Fano threefold $X$ as the integer $g$ such that
$\deg(X)=H_X^3=2\, g-2$.

Smooth prime Fano threefolds are classified up to deformation, see for instance
\cite[Chapter IV]{fano-encyclo}. The number of deformation classes is
$10$. The genus of a smooth non-hyperelliptic prime Fano
threefolds takes values in $\{3,4,\ldots,9,10,12\}$, while
there are two families (one for $g=2$, another for $g=3$) that consist
of hyperelliptic threefolds.
A hyperelliptic prime Fano threefold of genus $3$ is a flat
specialization of a quartic hypersurface in $\p^4$, see e.g.
\cite{manin:notes} and references therein.
It is well-known that a smooth non-hyperelliptic
prime Fano threefold is ACM.

Any prime Fano threefold $X$ contains lines and conics.
The Hilbert scheme $\sH^0_1(X)$ of lines contained in $X$ is a
projective curve.
It is well-known that the surface
swept out by the lines of a prime Fano threefold $X$ is linearly
equivalent to the divisor $rH_X$, for some $r\ge2$, see
e.g. \cite[Table at page 76]{fano-encyclo}.
Moreover, if $g\geq 4$, every line meets finitely many other lines, see
\cite[Theorem 3.4, iii]{iskovskih:II} (see also \cite{manin:notes}).
If $g=3$, then we know by \cite[Section 7]{harris-tschinkel} that
there always exist two disjoint lines in $X$.  

A prime Fano threefold $X$ is said to be {\it exotic} if the Hilbert
scheme $\sH^0_1(X)$ has a component which is non-reduced
at any point. 
By \cite[Lemma 3.2]{iskovskih:II}, this is equivalent to the fact that
for any line $L\subset X$ of this component, the normal bundle $N_L$ splits as $\OO_L(1) \oplus
\OO_L(-2)$. It is well-known that a general prime Fano threefold $X$
is not exotic. On the other hand, for $g\geq 9$, the results of
\cite{gruson-laytimi-nagaraj} and \cite{prokhorov:exotic} imply
that $X$ is non-exotic unless $g=12$ and $X$ is the Mukai-Umemura
threefold, see \cite{mukai-umemura}.
We say that a prime Fano threefold $X$ is {\it ordinary}
if the Hilbert scheme $\sH^0_1(X)$ has a generically smooth component.
This is equivalent to the fact that there exists a line $L\subset X$
whose normal bundle $N_L$ splits as $\OO_L \oplus \OO_L(-1)$.

If $X$ is a smooth non-hyperelliptic prime Fano threefold, the Hilbert scheme $\sH^0_2(X)$ of
conics contained in $X$ is a projective surface, and a general conic $C$ in $X$
has trivial normal bundle, see \cite[Proposition 4.3 and Theorem
4.4]{iskovskih:II}. Moreover, the threefold $X$ is covered by conics.
Moreover if $X$ is a general prime Fano threefold, then $\sH^0_2(X)$
is a smooth surface, see \cite{iliev-manivel:prime:math-ann} for a survey.

A smooth projective surface $S$ is a {\it K3 surface} if
it has trivial canonical bundle and irregularity zero.
A general hyperplane section $S$ of a prime Fano threefold $X$ is a K3
surface, polarized by the restriction $H_S$ of $H_X$ to $S$.
If $X$ has genus $g$, then $S$ has (sectional) genus $g$, and degree $H_S^2=2\,g-2$.
If $X$ is non-hyperelliptic, by Mo{\u\i}{\v{s}}ezon  's theorem \cite{moishezon:algebraic-homology},
we have $\Pic(S)\cong \Z = \langle H_S \rangle$.

Note that a general hyperplane section of a hyperelliptic prime Fano
threefold is still a K3 surface of Picard number $1$ if $g=2$.
This is no longer true in the other hyperelliptic case, i.e. for
$g=3$. Indeed, let  $X$ be a double cover of a smooth quadric in
$\p^4$ ramified on a general octic surface.
Then the general hyperplane section is a K3 surface of Picard number $2$.

\subsection{Summary of basic formulas}

From now on, $X$ will be a smooth prime Fano threefold of genus $g$,
polarized by $H_X$.
Let $F$ be a sheaf of (generic) rank $r$ on $X$ with Chern classes $c_1,c_2,c_3$.
Recall that these classes will be denoted by integers, since
$\HH^{k,k}(X)$ is generated by the class of $H_X$ (for $k=1$), the
class of a line contained in $X$ (for $k=2$), the class of a closed point of
$X$ (for $k=3$). We will say that $F$ is an $r$-bundle if it is
a vector bundle (i.e. a locally free sheaf) of rank $r$.
The {\it discriminant} of $F$ is defined as:
\begin{equation}
  \label{delta}
\Delta(F) = 2\,r\,c_2 - \,(r-1)(2\,g-2)\,c_1^2.
\end{equation}

Bogomolov's inequality, see for instance \cite[Theorem 3.4.1]{huybrechts-lehn:moduli},
states that if $F$ is a $\mu$-semistable sheaf, then we have:
\begin{equation}
\label{eq:bogomolov}
\Delta(F)\geq 0.
\end{equation}

Applying Hirzebruch-Riemann-Roch to $F$ we get:
\begin{align*}
\chi(F) & = r + \frac{11+g}{6}\,c_1 +\frac{g-1}{2}\, c_1^2
-\frac{1}{2} c_2 +  \frac{g-1}{3}\,c_1^3
-\frac{1}{2}\,c_1\,c_2+\frac{1}{2}\,c_3, 
\end{align*}

We recall by \cite{mukai:symplectic} (see also \cite[Part II, Chapter 6]{huybrechts-lehn:moduli}) that,
given a simple sheaf $F$ of rank $r$ on a K3 surface $S$ of genus $g$,
with Chern classes $c_1,c_2$,
the dimension at the point $[F]$ of the moduli space $\Spl_S$ 
is:
\begin{equation}
  \label{eq:dimension}
  \Delta(F) - 2\, (r^2-1),
\end{equation}
where $\Delta(F)$ is still defined by \eqref{delta}.
If the sheaf $F$ is stable, this value also equals the dimension
at the point $[F]$ of the moduli space $\Mo_S(r,c_1,c_2)$.

Let us focus on vector bundles $F$ of rank $2$.
Then we have $F \cong F^*(c_1(F))$.
Further, the well-known Hartshorne-Serre correspondence relates
vector bundles of rank $2$ over $X$ with subvarieties $Z$ of $X$ of
codimension $2$.
We refer to
\cite{hartshorne:small-codimension},
\cite{hartshorne:stable-vector-bundles-P3},
in particular \cite[Theorem 4.1]{hartshorne:stable-reflexive}
(see also
\cite{arrondo:home-made} for a survey).

\begin{prop}\label{hartshorneserre}
Fix the integers
$c_1,c_2$. Then we have a one-to-one correspondence between the sets
\begin{enumerate}
\item  of equivalence classes of pairs $(F,s)$, where $F$ is a rank
  $2$ vector bundle on $X$ with $c_i(F)=c_i$ and $s$ is a global
  section of $F$, up to multiplication by a non-zero scalar, whose
  zero locus has codimension $2$,  
\item of locally complete intersection 
curves $Z\subset X$ of degree $c_2$, with $\omega_Z\cong\OO_Z(c_1-1)$.
\end{enumerate}
\end{prop}

Recall that in the above correspondence $Z$ has arithmetic genus $p_a(Z)=1-\frac{d(1-c_1)}{2}$.
The zero locus of a non-zero global section $s$ of a rank $2$ vector bundle $F$
has codimension $2$ if $F$ is globally generated and $s$ is general,
or if $\HH^0(X,F(-1))=0$.

\begin{lem} \label{trebo}
Assume that $X$ is not hyperelliptic and let $F$ be a rank $2$ bundle
on $X$. Let $s$ be a global section of $F$,
whose zero locus is a curve $D\subset X$.
Then we have:
\begin{equation}
  \label{OO}
  \HH_*^1(X,F) \cong \HH_*^1(X,\cI_{D,X}(c_1(F))).
\end{equation}

In particular $F$ is ACM if and only if $\HH_*^1(X,\cI_{D,X})=0$, if
and only if $\HH_*^1(X,\cI_{D,\p^m})=0$.
If $D$ is smooth, $F$ is ACM if and only if $D$ is projectively normal.
\end{lem}

\begin{proof}
  The section $s$ gives the exact sequence:
    \begin{equation} \label{taragione}
      0 \to \OO_X \to F \to \cI_D(c_1(F)) \to 0,
    \end{equation}
    and taking cohomology we obtain the required isomorphism \eqref{OO}.
    By Serre duality, since $F$ is locally free, the condition
    $\HH_*^1(X,F)=0$ is equivalent to $\HH_*^2(X,F)=0$, and thus to
    $F$ being ACM. Since by \eqref{OO} the module
    $\HH_*^1(X,\cI_D)$ is isomorphic to $\HH_*^1(X,F)=0$ up to the grading,
    we have $\HH_*^1(X,\cI_D)=0$ iff $F$ is ACM.

    Take now $X\subset \p^m$ polarized by
    $H_X$ (which is very ample by assumption) and consider $D \subset X$.
    We have the exact sequence:
    \[
    0 \to \cI_{X,\p^m} \to \cI_{D,\p^m} \to \cI_{D,X} \to 0.
    \]
    Recall that $X$ is an ACM variety of dimension $3$, so that
    $\HH_*^k(X,\cI_{D,\p^m})=0$ for $k=1,2$.
    Therefore, taking cohomology in the above sequence, it follows
    that $\HH_*^1(X,\cI_{D,X})=0$ if and only if $\HH_*^1(X,\cI_{D,\p^m})=0$.

    Finally, if $D \subset \p^m$ is smooth, the condition
    $\HH_*^1(X,\cI_{D,\p^m})=0$ is equivalent to $D$ being
    projectively normal, see \cite[Chapter II,
    exercise 5.14]{hartshorne:ag}.
\end{proof}

\subsection{ACM bundles of rank $2$}
In this section, we recall Madonna's result in the case of bundles of
rank $2$ on a smooth prime Fano threefold.

\begin{prop}[Madonna, \cite{madonna:fano-cy}] \label{madonna}
  Let $F$
  be a normalized ACM $2$-bundle on $X$.
  Then the Chern classes $c_1$ and $c_2$ of $F$ satisfy the following restrictions:
  \begin{align*}
    & c_1 = 0  \Rightarrow c_2 \in \{2,4\}, \\
    & c_1 = 1  \Rightarrow c_2 \in \{1,\ldots,g+3\}.
  \end{align*}
\end{prop}

\begin{rmk}\label{pappo} 
  Let $F$, $c_1$, $c_2$ be as above, and $t_0$ be the smallest integer
  $t$ such that $\HH^0(X,F(t))\neq 0$.
  In \cite{madonna:fano-cy} the author computes the following values
  of $t_0$:
   \begin{enumerate}[(a)]
  \item \label{1} if $(c_1,c_2) = (1,1)$, then $t_0 = -1,$
  \item \label{2} if $(c_1,c_2) = (0,2)$, then $t_0 = 0,$
  \item \label{3} if $(c_1,c_2) = (1,c_2)$, with $2\leq c_2 \leq g+2$, then $t_0 = 0$,
  \item \label{4} if $(c_1,c_2) = (0,4)$, then $t_0 = 1$,
  \item \label{5} if $(c_1,c_2) = (1,g+3)$, then $t_0  = 1$.
  \end{enumerate}

  We observe that $F$ is not semistable in cases \eqref{1} and \eqref{2},
  and strictly $\mu$-semistable in case \eqref{2}.
  On the other hand, in the remaining cases, if $F$ exists then it is
  a $\mu$-stable sheaf.

  The existence of $F$ in cases \eqref{1} and \eqref{2} is well-known.
  It amounts to the existence of lines and conics 
  contained in $X$, in view of Proposition \ref{hartshorneserre}.
\end{rmk}

The following lemma (see \cite[Lemma 3.1]{brambilla-faenzi:genus-7})
gives a sharp lower bound on the values in Madonna's list.
We set:
  \begin{equation} \label{mg}
  m_g=\left\lceil{\frac{g+2}{2}}\right\rceil.
  \end{equation}

\begin{lem}
  The moduli space $\Mo_X(2,1,d)$ is empty for $d < m_g$.
  In particular, we get the further restriction $c_2 \geq m_g$
  in case \eqref{3}.
\end{lem}

\begin{rmk}
The assumption non-hyperelliptic cannot be dropped. Indeed if $X$ is a
hyperelliptic Fano threefold of genus $3$, then
the moduli space $\Mo_X(2,1,2)$ is not empty.
Indeed let $Q\in\p^4$ be a smooth quadric and $\pi:X\to Q$ be a
double cover ramified along a general octic surface. 
Set $F=\pi^{*}(\cS)$, where $\cS$ is the spinor bundle on $Q$. 
Then $F$ is a stable vector bundle on $X$ lying in
$\Mo_X(2,1,2)$. Notice that the restriction $F_S$ to any hyperplane
section $S\subset X$ is decomposable (hence strictly semistable).
\end{rmk}

\section{Bundles with odd first Chern class} \label{sec:odd}

Throughout the paper, we denote by $X$ a smooth non-hyperelliptic prime Fano threefold of genus $g$.
In this section we will prove the existence of the semistable bundles
appearing in the (restricted) Madonna's list, whose first Chern class is odd.
The main result of this section is the following existence theorem.
\begin{thm} \label{topolino}
  Let $X$ be a smooth non-hyperelliptic prime Fano threefold
  of genus $g$, and let $\frac{g}{2}+1 \leq d \leq g+3$.
  If $d \geq \frac{g}{2}+2$, we assume, in addition, that $X$ is ordinary.
  Then there exists an ACM vector bundle $F$ of rank $2$ with
  $c_1(F)=1$ and $c_2(F)=d$. Moreover, in the range $d\ge\frac{g}{2}+2$, 
  such a bundle $F$ can be chosen from a generically smooth component 
  of the moduli space $\Mo_X(2,1,d)$ of dimension $2d-g-2$.
\end{thm}

We will study first the case of minimal $c_2$ and then we proceed
recursively. 

\subsection{Moduli of ACM $2$-bundles with minimal $c_2$}
In this section we study the moduli space of rank $2$ semistable
sheaves with odd $c_1$ (we may assume that $c_1$ is $1$) and minimal $c_2$.
Namely, given a smooth non-hyperelliptic 
prime Fano threefold $X$ of genus $g$,
we set $m_g=\left\lceil{\frac{g+2}{2}}\right\rceil$ according to 
\eqref{mg}, and we study $\Mo_X(2,1,m_g)$.
Our goal is to prove:

\begin{thm} \label{riassuntone}
  Let $X$ be a smooth non-hyperelliptic prime Fano threefold of genus $g$.
  Then any sheaf $F$ lying in $\Mo_X(2,1,m_g)$ is locally free and ACM, and
  it is globally generated if $g\geq 4$.

  Further, there is a line $L \subset X$ such that:
  \begin{equation}
    \label{pippo-3}
    F\ts \OO_L \cong \OO_L \oplus \OO_L(1),
  \end{equation}
  and  $\Mo_X(2,1,m_g)$ can be described as follows:
  \begin{enumerate}[i)]
  \item \label{g3} the curve $\sH^0_1(X)$ parametrizing lines contained in $X$ if $g=3$;
  \item \label{g4} a scheme of length two if $g=4$, smooth if and only if $X$ is contained in a smooth
    quadric;
  \item \label{g5} a double cover of the Hesse septic curve if
    $g=5$, see \eqref{g=5} below;
  \item \label{6-12} a single smooth point if $g=6,8,10,12$; 
  \item \label{7} a smooth non-tetragonal curve of genus $7$ if $g=7$;
  \item \label{9} a smooth plane quartic if $g=9$.
  \end{enumerate}

  Moreover, if we assume that $X$ is ordinary if $g=3$ and that $X$ is
  contained in a smooth quadric if $g=4$, then there is a sheaf $F$ in
  $\Mo_X(2,1,m_g)$ with: 
  \begin{equation}
    \label{pippo-1}
    \Ext^2_X(F,F)=0.
  \end{equation}


Finally, if $X$ is ordinary, then the line $L$ in \eqref{pippo-3} can
be chosen in such a way that  $N_L \cong  \OO_L \oplus \OO_L(-1)$. 

\end{thm}

The proof of the above theorem is distributed along the following paragraphs.

\subsubsection{Non-emptiness}

It is well-known that, for any non-hyperelliptic smooth prime Fano
threefold $X$ of
genus $g$, the moduli space $\Mo_X(2,1,m_g)$ is non-empty.
Up to the authors' knowledge, there is no proof of this fact, other
than a case-by-case analysis. We refer e.g. to \cite{madonna:quartic} for
$g=3$, \cite{madonna:fano-cy} for $g=4,5$, \cite{gushel:fano-6} for $g=6$,
\cite{iliev-markushevich:genus-7},
\cite{iliev-markushevich:sing-theta:asian}, \cite{kuznetsov:v12}, for
$g=7$, \cite{gushel:fano-8-I}, \cite{gushel:fano-8-II}, \cite{mukai:biregular} for $g=8$,
\cite{iliev-ranestad} for $g=9$, \cite{mukai:biregular} for $g=10$,
\cite{kuznetsov:v22} (see also \cite{schreyer:V22}, \cite{faenzi:v22}) for $g=12$.

Given a sheaf $F$ in $\Mo_X(2,1,m_g)$,
we note that $F$ is locally free and $\HH^k(X,F)=0$ for $k\geq 1$ by
\cite[Proposition 3.5]{brambilla-faenzi:genus-7}. Riemann-Roch implies:
\[
\hh^0(X,F)=g+3-m_g,
\]
and any section $s\neq 0$ in $\HH^0(X,F)$
vanishes along a curve $C_s$, giving rise to the exact sequence:
\begin{align}
  & \label{Cs}
  0 \to \OO_X \overset{s}{\rr} F \to \cI_{C_s,X}(1) \to 0,
  \intertext{where $C_s$ has degree $m_g$. We immediately have:}
  & \label{printemps}
  \hh^0(X,\cI_{C_s,X}(1))=g+2-m_g.
\end{align}

\subsubsection{Cases $g\ge6$} 
Let $F$ be a sheaf in $\Mo_X(2,1,m_g)$ (there is such $F$ by the
previous paragraph).
From \cite[Proposition 3.5]{brambilla-faenzi:genus-7} it follows that
$F$ is locally free, ACM and globally
generated.
Given any line $L$ contained in $X$, the sheaf $F$ satisfies
\eqref{pippo-3}, indeed $F$ has degree $1$ and is globally generated
on $L$. Clearly we can choose $L$ with $N_L \cong \OO_L \oplus
\OO_L(-1)$ if $X$ is ordinary.

It only remains to study the structure of $\Mo_X(2,1,m_g)$.
We will do this with the aid of the following two lemmas, which 
are probably well-known to experts, but for which we ignore an
explicit reference in the literature.

\begin{lem} \label{even}
  Assume $g \geq 6$, and let $F$
  and $F'$ be sheaves in $\Mo_X(2,1,m_g)$.
  Then we have $\Ext^2_X(F,F')=0$. In particular the space
  $\Mo_X(2,1,m_g)$ is smooth.
  If $g$ is even, this implies that $\Mo_X(2,1,m_g)$ consists of a
  single smooth point.
\end{lem}

\begin{proof}
  We have said that $F'$ is globally generated, 
  so we write the natural exact sequence:
  \begin{equation} \label{evaluation} 
    0 \to K \to \HH^0(X,F') \ts \OO_X \xr{e_{\OO,F'}} F' \to 0, 
  \end{equation}
  where the sheaf $K$ is locally free, and we have:
  \[
  \rk(K)=g-m_g+1, \qquad c_1(K)=-1.
  \]
  
  Note that $K$ is a stable bundle by Hoppe's criterion, see for
  instance \cite[Theorem 1.2]{ancona-ottaviani:special}, \cite[Lemma
  2.6]{hoppe:rang-4}.
  Indeed, note that $\HH^0(X,K)=0$, and we have $-1 < \mu(\wedge^p K) < 0$,
  for $0<p<\rk(K)$.   
  By the inclusion:
  \[
  \wedge^p K \mono \wedge^{p-1} K \ts \HH^0(X,F),
  \]
  we obtain recursively $\HH^0(X,\wedge ^ p K)=0$ for all $p\ge 0$.
  
  Now, since $F$ is stable and ACM, we have $\HH^k(X,F^*)=0$
  all $k$.
  Thus, tensoring \eqref{evaluation} by $F^*$, we obtain:
  \[
  \ext^2_X(F,F') = \hh^2(X,F^* \ts F') = \hh^3(X,F^* \ts K)
  = \hh^0(X,K^* \ts F^*) = 0,
  \]
  where the last equality holds by stability, indeed $c_1(K^* \ts F^*) = m_g-g+1
  <0$ for $g\geq 6$.
  
  Note that, when $g$ is even, we have
  $\chi(F,F')=1$. This gives $\Hom_X(F,F') \neq 0$. But a non-zero
  morphism $F \to F'$ has to be an isomorphism. This concludes the proof.
\end{proof}

\begin{lem} \label{bella}
  Assume $g\geq 6$ and $g$ odd.
  Then the space $\Mo_X(2,1,m_g)$ is fine and isomorphic to a smooth irreducible curve.
\end{lem}

\begin{proof}
  Let $F$ be a sheaf in $\Mo_X(2,1,m_g)$.
  By Lemma \ref{even}, the moduli space is smooth and, since $\chi(F,F)=0$,
  we have $\ext^1_X(F,F)=\hom_X(F,F)=1$.
  Thus $\Mo_X(2,1,m_g)$ is a nonsingular curve.

  It is well known, from classical results due to Narasimhan, Ramanan
  and Grothendieck, that the obstruction to the existence of a universal
  sheaf on $X \times \Mo_X(2,1,m_g)$ corresponds to an element of the
  Brauer group of $\Mo_X(2,1,m_g)$. 
  But this group vanishes as soon as the variety $\Mo_X(2,1,m_g)$ is a
  nonsingular curve (see \cite{grothendieck:brauer}, see also
  \cite{caldararu:nonfine}).  
  Hence we have a universal vector bundle on $X
  \times \Mo_X(2,1,m_g)$.
  We consider a component $\Mo$ of $\Mo_X(2,1,m_g)$ and we let
  $\sE$ be the restriction of the universal sheaf to $X \times
  \Mo$. We let $p$ and $q$  be the  projections of $X \times \Mo$ respectively to $X$ and $\Mo$.

  To prove the irreducibility of $\Mo_X(2,1,m_g)$,
  we denote by $\EE_y$ the restriction of $\sE$ to $X \times \{y\}$.
  We have $\sE_y \cong \sE_z$ if and only if $y=z$, for $y,z\in
  \Mo$. Moreover, for any sheaf $F$ in $\Mo_X(2,1,m_g)$, we have:
  \[
  \Ext^k_X(\sE_y,F) = 0, \qquad \mbox{for $k=2,3$, and for all $k$ if
    $F \not \cong \sE_y$},
  \]
  where the vanishing for $k=2$ follows from Lemma \ref{even}.
  Hence we have:
  \begin{align*}
  & \RR^kq_*(\sE^* \ts p^*(F)) = 0, && \mbox{for $k\neq 1$},\\
  & \RR^1q_*(\sE^* \ts p^*(F)) \cong \OO_y, && \mbox{for $F \cong \sE_y$}.
  \end{align*}

 In particular, for any sheaf $F$ in $\Mo_X(2,1,m_g)$,
 the sheaf $\RR^1q_*(\sE^* \ts p^*(F))$ has rank zero and we have
 $\chi(\RR^1q_*(\EE^* \ts p^*(F)))=1$, for this value can be computed by
 Grothendieck-Riemann-Roch formula. Thus there must be a point $y \in \Mo$
 such that $\Ext^1_X(\sE_y,F)\neq 0$, hence  $\Hom_X(\sE_y,F)\neq 0$,
 so $F \cong \sE_y$. This implies  that $F$ belongs to $\Mo$.
 \end{proof}

 Lemma \ref{even} thus proves \eqref{pippo-1} as well as \eqref{6-12} of Theorem \ref{riassuntone}.
 The irreducibility statement of Lemma \ref{bella}
 proves that $\Mo_X(2,1,m_g)$ is a curve of the desired type by 
 \cite{iliev-markushevich:genus-7} for $g=7$ (in this case
 irreducibility was already known), and by \cite{iliev-ranestad} for
 $g=9$. 
 Theorem \ref{riassuntone} is thus proved for $g\geq 6$, and it
 remains to establish it for $g=3,4,5$, which we will do in the
 following three paragraphs.

 \subsubsection{Case $g=3$}
 A smooth non-hyperelliptic prime Fano threefold of genus $3$ is a
 smooth quartic threefold in $\p^4$.
 To prove Theorem \ref{riassuntone} we need the following proposition.
 
 \begin{prop}
   Let $X\subset\p^4 $ be a smooth quartic threefold.
   Then any element $F$ in $\Mo_X(2,1,3)$ is an ACM bundle and fits
   into an exact sequence of the form:
   \begin{equation}
     \label{blocus}
     0 \to \OO_X(-1) \to \HH^0(X,F)\ts \OO_X \xr{e_{\OO,F}}  F \to \OO_L(-2) \to 0, 
   \end{equation}
   where $L$ is a line contained in $X$.
   
   The map $F \mapsto L$ gives an isomorphism of
   $\Mo_X(2,1,3)$ to $\sH^0_1(X)$.
 \end{prop}
 
 \begin{proof} We consider a sheaf $F$ in $\Mo_X(2,1,3)$ and a cubic curve $C_s$ 
   associated to a non-zero global section $s$ by \eqref{Cs}.
   By \eqref{printemps}, the curve $C_s$ spans a
   projective plane $\Lambda \subset \p^4$ which intersects $X$ along
   $D= C_s \cup L$, where $L$ is a line.
   Then we have an exact sequence:
   \[
   0 \to \cI_{D,X}(1) \to \cI_{C_s,X}(1) \to \cI_{C_s,D}(1) \to 0.
   \]
   Note that $\cI_{C_s,D}$ is a torsion-free sheaf of rank $1$ supported on
   $L$, hence of the form $\OO_L(t)$. By calculating Chern classes,
   one easily shows $t=-3$, so the above exact sequence reads:
   \[
   0 \to \cI_{D,X}(1) \to \cI_{C_s,X}(1) \to \OO_{L}(-2) \to 0.
   \]
   Since $D$ is cut by two hyperplanes, we also
   have a surjective map $\OO_X^2 \to \cI_{D,X}(1)$ whose kernel is
   $\OO_X(-1)$.  
   It is easy to patch these exact sequences
   together with \eqref{Cs} to obtain a long exact sequence of the form:
   \begin{equation}
     \label{blocus2}
     0 \to \OO_X(-1) \to \OO_X^3 \to  F \to \OO_L(-2) \to 0,        
   \end{equation}
   which amounts to \eqref{blocus} for $\OO_L(-2)$ has no non-zero global
   sections.

   A straightforward Hilbert polynomial computation, and the remark
   that all sheaves in $\Mo_X(2,1,3)$ are stable, show
   that we can apply \cite[Corollary 4.6.6]{huybrechts-lehn:moduli}
   to get a universal sheaf $\EE$ on $X \times \Mo_X(2,1,3)$.
   We denote by $p$ and $q$ the projections from $X \times
   \Mo_X(2,1,3)$ respectively to $X$ and $\Mo_X(2,1,3)$.
   We can thus globalize the exact sequence \eqref{blocus2} over 
   $X \times \Mo_X(2,1,3)$ and write the middle arrow as the fiber
   over a point of $\Mo_X(2,1,3)$ of the natural map:
   \[
   q^*(q_*(\EE)) \to \EE.
   \]
   Taking the support of the cokernel sheaf of the above map we get a
   family of lines contained in $X$, parametrized by $\Mo_X(2,1,3)$,
   hence, by the universal property of the Hilbert scheme, this family is induced by a morphism
   $\alpha : \Mo_X(2,1,3) \to \sH^0_1(X)$.

   We observe that, dualizing and twisting \eqref{blocus2}, we easily obtain an
   exact sequence of the form:
   \begin{equation}
     \label{noblocus}
     0 \to F \to \OO_X^3(1) \to \cI_{L,X}(2) \to 0.     
   \end{equation}

   Let now $L$ be any line contained in $X$. Since $L$ is cut by
   three hyperplanes, we have a projection $\OO_X^3(1) \to
   \cI_{L,X}(2)$. It is easy to see that the kernel of this projection is a 
   stable bundle lying in $\Mo_X(2,1,3)$. 
   In order to globalize \eqref{noblocus}, we denote by $\sI$ the universal ideal sheaf on
   $X \times \sH^0_1(X)$, and by $f$ and $g$ the projections from 
   $X \times \sH^0_1(X)$ respectively to $X$ and $\sH^0_1(X)$.
   Thus we have a surjective map:
   \[
   f^*(\OO_X(1)) \ts g^*\left(g_*\left( \sI \ts  f^*(\OO_X(1))\right)\right) \to \sI \ts  f^*(\OO_X(2)).
   \]
   Therefore we have a family of sheaves in $\Mo_X(2,1,3)$
   parametrized by $\sH^0_1(X)$, hence a classifying map $\beta :
   \sH^0_1(X) \to \Mo_X(2,1,3)$.
   Since $\alpha$ and $\beta$ are mutually inverse, the schemes
   $\Mo_X(2,1,3)$ and $\sH^0_1(X)$ are isomorphic.
 \end{proof}
   
   Let us note that the above analysis implies
   Theorem \ref{riassuntone} for $X$.
   We know that $F$ is locally free and
   that the curve $C_s$ is a complete intersection in $\p^4$.
   Therefore we have $\HH^1_*(\p^4,\cI_{C_s,\p^4})=0$
   so that $F$ is ACM by Lemma \ref{trebo}.
   Condition \eqref{pippo-1} holds for $F$ as soon as $F$ corresponds to
   a smooth point of $\sH^0_1(X)$, and such points exist as soon as $X$ is
   ordinary. Finally, let $L' \subset X$ be a line which does not meet
   $L$ (it exists for any $X$ by \cite[Section 7]{harris-tschinkel}).
   Restricting \eqref{blocus2} to $L'$, we see that the splitting required for \eqref{pippo-3} holds on $L'$.

\subsubsection{Case $g=4$}
  A smooth prime Fano threefold $X$ of genus $4$ must be the complete
  intersection of a quadric $Q$ and a cubic in
  $\p^5$.
  Almost all the results we need for the next proposition follow from \cite[Section 3.2]{madonna:fano-cy}.
  
\begin{prop} \label{casog=4}
  Let $X$ be a smooth prime Fano threefold of genus $4$, let
  $Q\subset \p^5$ be the unique quadric containing $X$.
  \begin{enumerate}[i)]
  \item \label{liscia} If the quadric $Q$ is smooth, then $\Mo_X(2,1,3)$ consists of
    two smooth points given by two globally generated stable ACM
    bundles $F_1$ and $F_2$. Moreover we have a natural exact
    sequence:
    \begin{equation}
      \label{blocage}
      0 \to F_i(-1) \to \OO^4_X \xr{e_{\OO,F_{i+1}}} F_{i+1} \to 0,
    \end{equation}
    where we take the indices mod $2$. 
  \item If the quadric $Q$ is singular, then $\Mo_X(2,1,3)$ consists of
    a length-two scheme supported at a point which corresponds to 
    a globally generated stable ACM bundles $F$ with:
    \begin{equation}
      \label{greve}
      \hom_X(F,F)=\ext^2_X(F,F)=1, \qquad \ext^1_X(F,F)=2,      
    \end{equation}
    and we have a natural exact sequence:
    \begin{equation} 
      \label{onbloque}
      0 \to F(-1) \to \OO^4_X \xr{e_{\OO,F}} F \to 0. 
    \end{equation}
  \end{enumerate}
\end{prop}

\begin{proof}
  Given a sheaf $F$ in $\Mo_X(2,1,3)$, we consider a cubic curve $C_s$
  arising as the zero locus of a global section of $F$. By \eqref{printemps}, the
  curve $C_s$ is contained in $3$ independent hyperplanes. Therefore
  $C_s$ spans a projective plane $\Lambda$ which must be contained in
  $Q$ by degree reasons. 
  The curve $C_s$ is thus a complete intersection in $\p^5$, so that
  $\HH^1_*(\p^5,\cI_{C_s,\p^5})=0$ and $F$ is ACM by Lemma \ref{trebo}.

  Assume now that $Q$ is nonsingular. Then one considers the bundles $F_1$ and $F_2$
  obtained restricting to $X$ the two non-isomorphic spinor bundles
  $\cS_1$ and $\cS_2$ on $Q$.
  Note that $F_1$ and $F_2$ are not isomorphic, with
  $\Ext^k_X(F_i,F_i)=0$ for $k\geq 1$.
  One can check this computing the vanishing of $\HH^k(Q,\cS_i\ts
  \cS_i^*(-3))$ for all $k$, which in turn follows from Bott's
  theorem. It is easy to deduce that the $F_i$'s are stable, and hence
  provide two smooth points of $\Mo_X(2,1,3)$.
  
  Note that $\Lambda$ arises as the zero locus of a section of $\cS_i$ for some
  $i$ so that $F$ is the restriction of $\cS_i$ to $X$.
  We have thus proved that $\Mo_X(2,1,3)$ consists of two smooth points.
  Finally restricting the exact sequence \eqref{spinorazzi} to $X$, we
  obtain \eqref{blocage}.
  This finishes the proof in case \eqref{liscia}.

\vspace{0.1cm}
  
  We consider now the case when $Q$ is singular, namely $Q$ is a cone
  with vertex $v$ over a smooth quadric $Q'$ contained in $\p^4
  \subset \p^5$.
  Here, we have one spinor bundle $\cS$ on $Q'$ which lifts to a
  rank $2$ sheaf $\tilde{F}$ on $Q$ which is locally free away from
  $v$. It is easy to check that, restricting $\tilde{F}$ to $X$ we get
  a stable bundle in $\Mo_X(2,1,3)$.

  A plane $\Lambda \subset Q$ must be the span of $v$ and a line $L$
  contained in $Q'$, and recall that $L$ arises as the zero locus of a
  global section of $\cS$. This easily implies that $\Lambda$ is the zero locus of a
  global section of $\tilde{F}$, so that $F$ is the restriction of
  $\tilde{F}$.
  Therefore $\Mo_X(2,1,3)$ is supported at a single point $[F]$. By
  specialization from the case \eqref{liscia}, it follows
  that $\Mo_X(2,1,3)$ is a scheme structure of length $2$ over $[F]$.

  Further, an exact sequence of the form \eqref{spinorazzi-unosolo}
  takes place on $Q'$. Lifting this sequence to $Q$ and
  restricting to $X$, we obtain the exact sequence \eqref{onbloque}.
  It is now easy to obtain \eqref{greve}, by applying the functor
  $\Hom_X(F,-)$ to \eqref{onbloque}, noting that $\chi(F,F)=0$ and
  using Serre duality.
\end{proof}

All the statements of Theorem \ref{riassuntone} are now proved for
$X$, except the splitting 
\eqref{pippo-3}.
But since $F$ and $F_i$ are globally generated, this holds for any line $L\subset X$.

\subsubsection{Case $g=5$} \label{g=5}

 Let us first recall some basic facts concerning prime Fano threefolds of genus $5$,
  for which we refer to \cite[Section 1.5]{beauville:prym-jacobienne}.
  The threefold $X$ is defined as the complete intersection of a net
  $\Pi$ of quadrics in $\p^{6}$, namely for each point $y$ of the
  projective plane $\Pi$ we have a quadric $Q_{y} \subset \p^{6}$.
  This defines a quadric fibration $f : \mathcal{X} \to \Pi$, where
  $\mathcal{X}$ is the set of pairs of points $(x,y) \in \p^6\times \Pi$,
  where the point $x$ lies in $Q_y$ and $f$ is the projection onto the
  second factor.
  The plane $\Pi$ contains the Hesse septic curve $\cH$ of singular
  quadrics, and $\cH$ is smooth away from finitely many ordinary
  double points.
  Each quadric $Q_{y}$ in $\cH$ has rank at least $5$ (for $X$ is smooth) and
  admits one or two rulings according to whether $\rk(Q_{y})$ equals
  $5$ or $6$. The curve parametrizing these rulings is denoted by
  $\tilde{\cH}$. It admits an involution $\tau$ whose only fixed
  points lie over the singularities of $\cH$, and we have $\cH \cong
  \tilde{\cH}/\tau$.
  This defines $\tilde{\cH}$ as a double cover of $\cH$, and we
  say that $\tilde{\cH}$ is associated to $X$.
  Further, 
  we consider the set of projective spaces $\p^{3} \subset Q_{y}$ belonging to the
  same ruling of $Q_{y}$.
  This defines a $\p^{3}$-bundle ${\sf G}(f) \to \tilde{\cH}$, and we
  denote by $\p^{3}_{\tilde{y}}$ the fibre over $\tilde{y} \in
  \tilde{\cH}$.

\begin{prop} 
  Let $X$ be a smooth prime Fano threefold of genus $5$, and let
  $\tilde{\cH}$ be associated to $X$.
  Then the space $\Mo_{X}(2,1,4)$ is isomorphic to $\tilde{\cH}$ 
  and any element $F\in\Mo_{X}(2,1,4)$ is globally generated and ACM.

  Moreover, there is an involution $\rho$ on
  $\Mo_{X}(2,1,4)$, which associates to $F$ the sheaf $F^\rho$ fitting into:
  \begin{equation}
    \label{invo}
    0 \to F^\rho(-1) \to \OO_X^4 \xr{e_{\OO,F}} F \to 0,     
  \end{equation}
  and $\rho$ corresponds to $\tau$ under the isomorphism
  $\Mo_{X}(2,1,4) \cong \tilde{\cH}$.
\end{prop}

\begin{proof}
  We will identify the fibre $\p^{3}_{\tilde{y}}$ of the
  $\p^{3}$-bundle ${\sf G}(f) \to \tilde{\cH}$ with the
  projectivized space of sections
  of a rank $2$ sheaf on $Q_y$.
  We distinguish the two cases
  $\rk(Q_y)=5,6$, for $y\in\cH$.
  
  If the quadric $Q_y$ has rank $6$, then it is a cone with vertex on a
  point $v$ over a smooth quadric $Q'_y$ in a $\p^5$ contained in
  $\p^6$.
  The set of projective three-spaces $\Lambda$
  contained in $Q_y$ is thus parametrized by the set of planes in
  $Q'_y$. These planes are in bijection with the elements of
  $\p(\HH^0(Q_y',\cS_i))$ for $i=1,2$   , where $\cS_1$, $\cS_2$ are the spinor bundles on $Q_y'$.
  Each of the bundles $\cS_i$ extends to a sheaf $\tilde{F}_i$ on $Q_y$
  which is locally free away from $v$, and we easily compute
  $\hh^0(Q_y,\tilde{F}_i)=\hh^0(Q'_y,\cS_i)=4$.
  We note incidentally that there is a natural exact sequence  (we take the indices mod $2$):
  \begin{equation}
    \label{tilde}
    0 \to \tilde{F}_i(-1) \to \OO_{Q_y}^4 \xr{e_{\OO,\tilde{F}_{i+1}}} \tilde{F}_{i+1} \to 0,    
  \end{equation}
  which lifts to $Q_y$ the exact sequence \eqref{spinorazzi}.

  Summing up, a subspace $\Lambda=\p^3$ contained in $Q_y$ corresponds to an element of $\p(\HH^0(Q_y,\tilde{F}))$,
  with $\tilde{F}=\tilde{F}_1$ or $\tilde{F}=\tilde{F}_2$.
  So $\p^3_{\tilde{y}}$ is canonically identified with $\p(\HH^0(Q_y,\tilde{F}))$.
  Note also that, given $\Lambda \subset Q_y$ we have:
  \begin{equation}
    \label{Lambda}
      0 \to \OO_{Q_{y}} \to \tilde{F} \to \cI_{\Lambda,Q_{y}}(1) \to 0.
  \end{equation}

  If $\rk(Q_y)=5$, then $Q_y$ is a cone with vertex on a line $L
  \subset \p^6$ over a smooth quadric $Q_y'' \subset \p^4$.
  In this case the subspaces $\Lambda = \p^3$ of $Q_y$ are given by
  lines in $Q_y''$, any of which is given as zero locus of a section
  of $\cS$, the spinor bundle on $Q_y''$. One can lift $\cS$ to a
  rank $2$ sheaf $\tilde{F}$ on $Q_y$ which is locally free away from
  $L$.
  We still have $\hh^0(Q_y,\tilde{F})=4$ and \eqref{Lambda} still
  holds, so that $\p^3_{\tilde{y}}$ is again identified with $\p(\HH^0(Q_y,\tilde{F}))$.

  In order to prove our statement, we show that the $\p^{3}$-bundle ${\sf G}(f)$ is isomorphic over
  the base curve to the $\p^3$-bundle on $\Mo_X(2,1,4)$ consisting of 
  pairs $([s],F)$, where $F$ lies in $\Mo_X(2,1,4)$ and $[s]$ lies in $\p(\HH^0(X,F))$.
  Recall that $s$ gives rise to the curve $C_s$ which has degree $4$.
  By \eqref{printemps}, the curve $C_s$ is contained in $3$
  independent hyperplanes, so $C_s$ spans a $\p^3$ which must be
  contained in a singular quadric $Q_y$. Note that this quadric is
  unique for otherwise $X$ would contain a quadric surface,
  contradicting $\Pic(X)\cong \langle H_X \rangle$.
  Note also that the curve $C_s$ is a complete intersection in $\p^6$,
  so $\HH^1_*(\p^6,\cI_{C_s,\p^6})=0$, hence $F$ is ACM by Lemma \ref{trebo}.

  Having this set up, we associate to $[s] \in \p(\HH^0(X,F))$ the
  element $[\tilde{s}]$ of $\p(\HH^0(Q_y,\tilde{F}))$ which
  corresponds to the space spanned by $C_s$.
  It is easy to see that $\HH^0(Q_y,\tilde{F})$ is naturally isomorphic to $\HH^0(X,F)$,
  and that restricting \eqref{Lambda} from $Q_y$ to $X$ we obtain \eqref{Cs}.
  Then $\p^3_{\tilde{y}}$ is identified with $\p(\HH^0(X,F))$ and we
  can associate $\tilde{y}$ to $F$.
  
  Note that this construction is reversible, namely to a point
  $\tilde{y}$ of $\tilde{\cH}$ we associate the bundle $F$ on $X$ such
  that, for any element $\Lambda$ in $\p^3_{\tilde{y}}$, the
  intersection $\Lambda \cap X$ is a curve of degree $4$ obtained as
  zero locus of a section of $F$. This proves that $\Mo_X(2,1,4)$ is
  isomorphic to $\tilde{\cH}$.
  The vanishing \eqref{pippo-1} holds for $F$ as soon as $F$
  corresponds to a smooth point of $\tilde{\cH}$.
  We remark that any such $F$ is globally generated, for it is the
  restriction to $X$ of $\tilde{F}$, which is globally generated
  (which is clear for instance by \eqref{Lambda}).

  It remains to check the statement regarding the involutions on $\Mo_X(2,1,4)$ and
  $\tilde{\cH}$. We have to check that $\rho$ is well-defined and
  that, under the above isomorphism, it agrees with $\tau$ which
  by definition interchanges the rulings of $Q_y$, as soon as
  $\rk(Q_y)=6$.
  Recall that any sheaf $F$ in $\Mo_X(2,1,4)$ is the
  restriction to $X$ of a sheaf $\tilde{F}$ on $Q_y$, say of
  $\tilde{F}_1$, which corresponds to one ruling of $Q_y$.
  We have thus \eqref{tilde} (with $i=0$), and restricting to $X$ we get an exact
  sequence of the form \eqref{invo}, for some sheaf $F^\rho$ lying in
  $\Mo_X(2,1,4)$.
  Note that $F^\rho$ is then the restriction to $X$ of the sheaf
  $\tilde{F}_2$ on $Q_y$.
  Since $\tilde{F}_2$ correspond to the second ruling of $Q_y$,
  we have proved that $\rho$ agrees with $\tau$.
\end{proof}

The above proposition proves Theorem \ref{riassuntone} for $X$, once we check 
\eqref{pippo-3}. But this splitting holds for any line $L \subset X$, since
any $F \in \Mo_X(2,1,4)$ is globally generated.

\vspace{0,1cm}

We have now finished the proof of Theorem \ref{riassuntone}.

\subsection{Moduli of ACM $2$-bundles with intermediate $c_2$}
This section is devoted to the proof of Theorem \ref{topolino} 
in the cases $m_g+1\le d\le g+2$. 
This will prove in particular the existence of case \eqref{3} 
of Madonna's list, see Remark \ref{pappo}. 
We will need a series of lemmas to prove recursively the existence of
ACM bundles of rank $2$.
The following one is proved in \cite[Theorem
3.12]{brambilla-faenzi:genus-7}, once we take care of the special case
of Fano threefolds of genus $4$ contained in a singular quadric.
Note that this is the only case when no sheaf $F$ in $\Mo_X(2,1,m_g)$
satisfies \eqref{smooth-odd}.

\begin{lem} \label{add-line}   
Let $X$ be ordinary 
and let $L$ be a general line in $X$.
Then, for any integer $d \geq m_g+1$, there exists a rank $2$ stable locally free sheaf
 $\f$ with $c_1(\f)=1$, $c_2(\f)=d$, and satisfying:
\begin{align}
\label{smooth-odd} & \Ext_X^2(\f,\f) = 0, \\
\label{noH1-odd} & \HH^1(X,\f(-1)) = 0, \\
\label{nosection-odd} & F\ts \OO_L \cong \OO_L \oplus \OO_L(1),
\end{align}
where $L$ is a line with $N_L \cong \OO_L \oplus \OO_L(-1)$.
\end{lem}

\begin{proof}
All statements are proved in \cite[Theorem
3.12]{brambilla-faenzi:genus-7}, by induction on $d\geq m_{g+1}$,
except when $g=4$ and $X$ is contained in a singular quadric.
The induction step goes as follows.
Given a stable $2$-bundle $F_{d-1}$ with $c_1(F_{d-1})=1$ and
$c_2(F_{d-1})=d-1$, satisfying \eqref{nosection-odd} for a given line
$L \subset X$ (with $N_L \cong \OO_L \oplus \OO_L(-1)$), 
we have a unique exact sequence:
\begin{equation} \label{defEd}
0 \rr \sS_d \rr \f_{d-1} \overset{\sigma}{\rr} \OO_L \rr 0,
\end{equation}
where $\sigma$ is the natural surjection and $\sS_d = \ker(\sigma)$ is a
non-reflexive sheaf in $\Mo_X(2,1,d)$.
We have proved in [{\it loc. cit.}, Theorem 3.12] that if
\eqref{smooth-odd} and \eqref{noH1-odd} hold for $F_{d-1}$,
then we get a vector bundle of $\Mo_X(2,1,d)$, satisfying 
\eqref{smooth-odd}, \eqref{noH1-odd} and \eqref{nosection-odd}
by flatly deforming $\sS_d$. 

Assume thus $g=4$ and that $X$ is contained in a singular quadric.
Given a line $L$ contained in $X$, such that $N_L \cong \OO_L \oplus \OO_L(-1)$, 
and the vector bundle $F_3\in\Mo_X(2,1,3)$ (see Proposition \ref{casog=4}), we set
$\sS_4=\ker (F_3\to\OO_L)$ (where the map
is non-zero). We obtain an exact sequence of the form \eqref{defEd}, with $d=4$.

The sheaf $\sS_4$ sits in $\Mo_X(2,1,4)$ and we want to prove 
$\Ext^2_X(\sS_4,\sS_4)=0$.
Applying $\Hom_X(\sS_4,-)$ to \eqref{defEd} we have:
\[\Ext^1_X(\sS_4,\OO_L)\to \Ext^2_X(\sS_4,\sS_4)\to\Ext^2_X(\sS_4,F_3).\]
It is easy to check that the first term vanishes by applying
$\Hom_X(-,
\OO_L)$ to \eqref{defEd}, and using 
\cite[Remark 2.1]{brambilla-faenzi:genus-7} and the fact that
\eqref{nosection-odd} holds for $F_3$.
To prove the vanishing of the last term we apply $\Hom_X(\sS_4,-)$ to
\eqref{onbloque} and we note that $\ext^3_X(\sS_4,F_3(-1))=\hom(F_3,\sS_4)=0$
by Serre duality and stability.

Having this set up, the sheaf $\sS_4$ admits a smooth neighborhood in
$\Mo_X(2,1,m_g+1)$, which has dimension $2$ in view of an easy
Riemann-Roch computation. On the other hand, the sheaves fitting in
\eqref{defEd} fill in a curve in  $\Mo_X(2,1,m_g+1)$ by \cite[Lemma
3.9]{brambilla-faenzi:genus-7}.
Therefore, the remaining part of the argument of
[{\it  loc. cit.}, Theorem 3.12] goes through.
\end{proof}

\vspace{-0.3cm}

\begin{dfn} \label{dfn:Md}
Let $X$ be ordinary. 
Let $\Mo(m_g)$ be a component of $\Mo_X(2,1,m_g)$ containing a stable
locally free sheaf $F$ satisfying the three conditions
\eqref{smooth-odd}, \eqref{noH1-odd} and \eqref{nosection-odd} (when
$g=4$ and $X$ is contained in a singular quadric we just set
$\Mo(3) = \{F\}$, with $F$ given by Proposition \ref{casog=4}).
This exists by Theorem \ref{riassuntone}, and coincides with $\Mo_X(2,1,m_g)$ for $g\ge6$.
For each $d\geq m_g+1$, we recursively define $\No(d)$ as the set of non-reflexive sheaves
$\sS_d$ fitting as kernel in an exact sequence of the form \eqref{defEd},
with $\f_{d-1} \in \Mo(d-1)$ (and $F_{d-1}$ satisfying \eqref{smooth-odd},
\eqref{noH1-odd} and \eqref{nosection-odd}), and $\Mo(d)$ as
the component of the moduli scheme $\Mo_X(2,1,d)$ containing $\No(d)$.
We have:
\[
\dim(\Mo(d))=2\,d-g-2.
\]
\end{dfn}

\begin{lem} \label{lem:sezioni-dispari}
Let $X$ be ordinary. 
For each $m_g \leq d \leq
g+2$, the general element $F_d$ of $\Mo(d)$ satisfies: 
\vspace{-0.3cm}
\begin{align*}
  & \hh^0(X,F_d) = g+3-d, \\
  & \HH^k(X,F_d)=0, && \qquad \mbox{for $k\geq 1$.}
\end{align*}
\end{lem}

\begin{proof}
The proof works by induction on $d$.
The first step of the induction corresponds to $d=m_g$, and follows from
Theorem \ref{riassuntone}.

Note that $\HH^3(X,F_{d})=0$ for all $d$ by Serre duality and
stability, and by Riemann-Roch we have $\chi(F_d) = g+3-d$.

Assume now that the statement holds for $F_{d-1}$ with $d \leq g+2$, and let us
prove it for a general element $\sS_d$ of $\No_X(d)$.
By semicontinuity the claim will follow for the general element $F_d\in\Mo(d)$.
So let $F_{d-1}$ be a locally free
sheaf in $\Mo(d-1)$.
By induction we know that $\hh^0(X,F_{d-1}) = g+3-d+1\geq 2$.
A non-zero global section $s$ of $F_{d-1}$ gives the exact sequence:
\vspace{-0.2cm}
\begin{equation}
  \label{eq:sezione}
  0 \to \OO_X \xr{s} F_{d-1} \to \cI_C(1) \to 0,
\end{equation}
where $C$ is a curve of degree $d-1$ and arithmetic genus $1$.
We want to show that we can choose a line $L\subset X$ and a section $s$ so that
$C$ does not meet $L$, and this will prove: 
\begin{equation}
  \label{soleil}
  \OO_L \ts \cI_{C}(1) \cong \OO_L(1).
\end{equation}

To do this, we note that $\hh^0(X,\cI_C(1))=g+3-d\geq1$, so $C$ is contained in some
hyperplane section surface $S$ given by a global section $t$ of $\cI_C(1)$.
Let $L$ be a general line such that $F_{d-1} \ts \OO_L \cong \OO_L
\oplus \OO_L(1)$ and $L$ meets $S$ at a single point $x$. We may
assume the latter condition because there exists a line in $X$ not
contained in $S$ (indeed, the lines
contained in $X$ sweep a divisor of degree greater than one,
see Section \ref{prime-fano}).
Then we write down the following exact commutative diagram:
\vspace{-0.3cm}
\[
\xymatrix@-2ex{
 & & 0 \ar[d] & 0\ar[d] \\
 & & \OO_X \ar@{=}[r] \ar^-{t}[d] & \OO_X \ar^-{t}[d] \\
0 \ar[r] & \OO_X \ar^-{s}[r] \ar@{=}[d] & F_{d-1} \ar^-{s^{\top}}[r]
\ar^-{t^{\top}}[d] & \cI_{C}(1) \ar[d]  \ar[r] &0\\
0 \ar[r] & \OO_X \ar^-{s}[r]\ar[d] & \cI_{D}(1) \ar[r]\ar[d] & \OO_{S}(H_S-C)
\ar[r]\ar[d]&0\\
 & 0 & 0 & 0&\\
}\]
which in turn yields the exact sequence:
\[
0 \to \OO_X^2 \xr{\binom{s}{t}} F_{d-1} \to \OO_{S}(H_S-C) \to 0,
\]
and dualizing we obtain:
\begin{equation}\label{eq:2sezioni}
0 \to F_{d-1}^* \xr{(s^{\top}\,t^\top)} \OO_X^2 \to \OO_{S}(C) \to 0.
\end{equation}

Thus the curve $C$ moves in a pencil without base points in the surface $S$, and each
member $C'$ of this pencil corresponds to a global section $s'$ of
$F_{d-1}$ which vanishes on $C'$. Therefore we can choose $s$ so that $C$
does not contain $x$.

Now let $\sigma$ be the natural surjection $F_{d-1} \to \OO_L$ and
$\sS_d = \ker(\sigma)$. We have thus the exact sequence
\eqref{defEd}. 
Taking cohomology, from induction hypotheses we obtain $\HH^2(X,\sS_d)=0$. 

By tensoring \eqref{eq:sezione} by $\OO_L$, in view of \eqref{soleil}, we see that
the composition $\sigma \circ s$  must be non-zero (in fact it is surjective).
Thus the section $s$ does not lift to $\sS_d$
so $\hh^0(X,\sS_d) \leq \hh^0(X,F_{d-1}) - 1$.
We have thus:
\[
\hh^0(\sS_d)\geq \chi(\sS_d) = \chi(F_{d-1})-1 = \hh^0(F_{d-1})-1 \geq  \hh^0(\sS_d),
\]
and our claim follows.
\end{proof}

Using Lemma \ref{trebo}, it is straightforward to deduce from the
previous lemma the following corollary. 
\begin{corol} \label{prenom}
Let $X$ be ordinary. 
  For $d \leq g+2$, let $D$ be the zero locus of a non-zero global
  section of a general element $F$ of $\Mo(d)$. Then we have:
\begin{align*}
  & \hh^0(X,\cI_D(1)) = g+2-d.
\end{align*}
\end{corol}

We are  now in position to prove Theorem \ref{topolino} in the cases when
$c_2\le g+2$.

\begin{proof}[Proof of Theorem \ref{topolino} for $d\le g+2$]
We work by induction on $d\ge m_g$.
By Theorem \ref{riassuntone}, the statement holds for $d=m_g$.

Assume now that $m_g<d\le g+2$.
By Lemma \ref{add-line} we can consider a general sheaf $F$ in $\Mo(d)$. 
Recall that $F$ is obtained as a general deformation of a sheaf $\sS_{d}$ 
fitting into an exact sequence of the form \eqref{defEd}, 
where $F_{d-1}$ is a vector bundle in $\Mo(d-1)$. 
By induction we assume that $F_{d-1}$ is ACM.
It remains to prove that $F$ is ACM too.

Since $d-1\le g+2$, as in the proof of Lemma \ref{lem:sezioni-dispari},
we can choose a line $L \subset X$, a projection $\sigma : F_{d-1} \to \OO_L$,
and a global section $s \in \HH^0(X,F_{d-1})$ such that
$\sigma \circ s$ is surjective. We can assume $\sS_d = \ker(\sigma)$. 
Let $C$ be the zero locus of $s$. 
Then we have the following exact diagram:
\begin{equation}
  \label{semplificazione}
  \xymatrix@-2ex{
    &  0 \ar[d] & 0 \ar[d] &&&\\
    0 \ar[r] &   \cI_L \ar[d] \ar[r] & \OO_X \ar[d]^-{s} \ar[r] & \OO_L \ar@{=}[d] \ar[r] &0\\
    0 \ar[r] &  \sS_{d} \ar[r] \ar[d] & F_{d-1} \ar[d] \ar^-{\sigma}[r] & \OO_L \ar[r]&0 \\
    &  \cI_{C}(1) \ar@{=}[r]\ar[d] & \cI_C(1) \ar[d] & &\\
    &  0  & 0  &&}
\end{equation}

Since $L$ is projectively normal, the leftmost column implies that 
$\HH_*^1(X,\sS_d) \subset \HH_*^1(X,\cI_{C}(1))$.
By Lemma \ref{trebo} we have $\HH_*^1(X,\cI_C(1)) \cong
\HH_*^1(X,F_{d-1})$, and this module vanishes by the induction hypothesis.
So we obtain $\HH^1_*(X,\sS_{d})=0$,
hence by semicontinuity the module $\HH^1_*(X,F)$ is zero as well. Then
by Serre duality the vector bundle $F$ is ACM.
\end{proof}

The following lemma will be needed later on.

\begin{lem}\label{ragione}
Let $D$ be the zero locus of a global section of a sheaf $F$ lying in
$\Mo_X(2,1,d)$ satisfying \eqref{smooth-odd}, \eqref{noH1-odd} and
such that $\HH^1(X,\cI_D(1))=0$.
Then we have $\Ext^2_X(\cI_D,\cI_D)=0$ and we obtain an exact sequence:
\begin{equation}
  \label{eq:taproprioragione}
  0 \to \HH^0(X,\cI_D(1)) \to \Ext^1_X(\cI_D,\cI_D) \to \Ext^1_X(F,F) \to 0,
\end{equation}
so $\ext^1_X(\cI_D,\cI_D)=d$.
\end{lem}

\begin{proof}
  We apply the functor $\Hom_X(F,-)$ to the exact sequence \eqref{taragione}. It is easy to check that
  $\Ext^k_X(F,\OO_X)=0$ for any $k$, thus we obtain for each $k$ an isomorphism:
  \[
  \Ext^k_X(F,\cI_D(1)) \cong \Ext^k_X(F,F).
  \]

  Therefore, applying $\Hom_X(-,\cI_D(1))$ to \eqref{taragione}, we get the vanishing
  $\Ext^2_X(\cI_D,\cI_D)=0$ and, since $F$ is a stable (hence simple) sheaf, we obtain
  the exact sequence \eqref{eq:taproprioragione}. The value of
  $\ext^1_X(\cI_D,\cI_D)$ can now be computed by Riemann-Roch.
\end{proof}

\subsection{Moduli of  ACM $2$-bundles with maximal $c_2$}
In order to complete the proof of Theorem \ref{topolino} we have to
consider the case $d=g+3$. 
This will give the existence of case \eqref{5} 
of Madonna's list. 
We need the following lemma.

\begin{lem}\label{stoccafisso}
  Let $F$ be a rank $2$ stable bundle on $X$ with $c_1(F)=1$. 
  Then $F$ is ACM if:
  \[
  \HH^k(X,F)=\HH^k(X,F(-1))=0, \qquad \mbox{for any $k$}.
  \]
\end{lem}

\begin{proof}
First we prove that $\HH^1(X,F(t))=0$ for any integer $t$.
By \cite[Remark 3.11]{brambilla-faenzi:genus-7} we deduce this
vanishing for any $t\le0$.

Let $S$ be a general hyperplane section of $X$. Taking cohomology of
the restriction exact sequence 
\begin{equation}\label{ladevomettere}
0\to F(-1+t)\to F(t)\to F_S(t) \to0,
\end{equation}
we obtain that $\HH^k(S,F_S)=0$ for any
$k$. By Serre duality since $F^*\cong F(-1)$, we also have
$\HH^k(S,F_S(-1))=0$ for any $k$. It follows that $\HH^k(C,F_C)=0$ for
any $k$, where $C$ is the general sectional curve of $X$.
Now since $\HH^0(C,F_C(t))=0$ for any $t\le0$, from the restriction
exact sequence 
\[
0\to F_S(-1+t)\to F_S(t)\to F_C(t) \to0,
\]
we deduce that $\HH^1(S,F_S(t))=0$ for any $t\le0$. By Serre duality
this also implies $\HH^1(S,F_S(t))=0$ for any $t\ge0$.  
Now from \eqref{ladevomettere} we obtain $\HH^1(X,F(t))=0$ for any
$t\ge0$, and we have proved $\HH^1_*(X,F)=0$. 
By Serre duality we immediately obtain the vanishing $\HH^2_*(X,F)=0$, and we
are done.  
\end{proof}

\begin{proof}[Proof of Theorem \ref{topolino} for $d=g+3$]
By Lemma \ref{add-line}, there exists a sheaf $F_{g+3}$ in $\Mo(g+3)$,
obtained as a general deformation of a sheaf $\sS_{g+3}$ fitting into
the exact sequence  
\begin{equation} \label{anchequesta}
0 \rr \sS_{g+3} \rr \f_{g+2} {\rr} \OO_L \rr 0,
\end{equation}
where $F_{g+2}\in \Mo(g+2)$ and $L$ is a line contained in $X$,
representing a smooth point of $\cH^0_1(X)$.
We already know that Theorem \ref{topolino} holds true for $c_2=g+2$,
hence we can assume that $F_{g+2}$ is ACM, so $\hh^0(X,F_{g+2})=1$.
It remains to prove that $F_{g+3}$ is ACM too.

We can assume that $F_{g+3}$ satisfies condition \eqref{noH1-odd},
since $\sS_{g+3}$ does, so by \cite[Remark 3.10]{brambilla-faenzi:genus-7} $\HH^k(X,F_{g+3}(-1))=0$
for all $k$. 
Taking cohomology of \eqref{anchequesta}, we get
$\HH^2(X,\sS_{g+3})=0$, hence by semicontinuity we can assume
$\HH^2(X,F_{g+3})=0$. On the other hand $\HH^3(X,F_{g+3})=0$, by Serre
duality and stability.

Note that $\hh^0(X,F_{g+3}) \leq 1$ by semicontinuity and \eqref{anchequesta}.
If $\hh^0(X,F_{g+3})=0$, then by Riemann-Roch formula 
we also have $\HH^1(X,F_{g+3})=0$. Then we can apply Lemma
\ref{stoccafisso} and we conclude that $F_{g+3}$ is ACM.

In order to complete the proof, we can now assume that there is an open dense neighborhood
$\Omega \subset \Mo_X(2,1,g+3)$ of the point
representing $\sS_{g+3}$ such that all elements $F_{g+3}$ (including
$\sS_{g+3}$) satisfy 
$\hh^0(X,F_{g+3})=1$, and show that this leads to a contradiction.
By Lemma \ref{add-line}, 
we can assume
$\Ext^2_X(F_{g+3},F_{g+3})=0$ which by Riemann-Roch implies:
\begin{equation}
  \label{ciprovo13}  \dim(\Omega)=\ext^1_X(F_{g+3},F_{g+3})=g+4.
\end{equation}
For any $F_{g+3}$ in $\Omega$ we consider the curve $D$
which is the zero locus of the (unique up to scalar) non-zero global
section of $F_{g+3}$.
This gives a map:
\[
\beta : \Omega \to \sH^1_{g+3}(X),
\]

We observe that the sheaf $F_{g+3}$ can be recovered from $D$ in view of
Proposition \ref{hartshorneserre}, so that $\beta$ is injective.
We will prove that $\sH^1_{g+3}(X)$ is smooth and locally of dimension
$g+3$ around the point representing $D$, which contradicts $\beta$
being injective since $\dim(\Omega)=g+4$.
In order to do this, we will prove:
 \begin{align}
   \label{ciprovo11}
   & \Ext^2_X(\cI_D,\cI_D)=0, 
   && \ext^1_X(\cI_D,\cI_D)=g+3,
 \end{align}
where the second equality follows from the first vanishing by
Riemann-Roch.

Consider now a non-zero global section $s$ of the (non-reflexive) sheaf $\sS_{g+3}$.
We will say that a curve $B \subset X$ is the zero locus of $s$ if we have an exact sequence:
\[
  0 \to \OO_X(-1) \xr{s} \sS_{g+3}(-1) \to \cI_{B} \to 0.
\]
Note that the section $s$ induces a (non-zero) global section of $F_{g+2}$, whose zero locus is a
curve $C \subset X$.
The exact sequence
 \begin{align}
   \label{ciprovo2}
   & 0 \to \OO_X(-1) \to F_{g+2}(-1) \to \cI_C \to 0
\intertext{induces, in view of \eqref{anchequesta} twisted by
  $\OO_X(-1)$, the exact sequences:}
   \label{ciprovo1}
   & 0 \to \cI_{C \cup L} \to \cI_{C} \to \OO_L(-1) \to 0, \\
   & 0 \to \OO_X(-1) \to \sS_{g+3}(-1) \to \cI_{C \cup L} \to 0,
 \end{align}
so $C \cup L$ is the zero locus of $s$.

Note that our neighborhood $\Omega$ 
gives a flat family of curves in $X$, namely at the point
corresponding to a sheaf $F$ we associate the zero locus of its
(unique up to scalar) non-zero global section.
The central fiber of this family
(the one corresponding to the sheaf $\sS_{g+3}$) is $C \cup L$,
while the general fiber is $D$, so that $D$ is a deformation of $C
\cup L$.
Then it will suffice to prove \eqref{ciprovo11} on $C\cup L$.
 The rest of the proof is devoted to this task.

 Applying the functor $\Hom_X(-,\OO_L(-1))$ to \eqref{ciprovo2} we obtain:
   \[
   0 \to \Hom_X(F_{g+2},\OO_L) \to \Hom_X(\OO_X,\OO_L)
   \to \Ext^1_X(\cI_C,\OO_L(-1)) \to \Ext^1_X(F_{g+2},\OO_L).
   \]
 Indeed we have:
 \begin{equation}
   \label{ciprovo6}
   \Hom_X(\cI_C,\OO_L(-1)) \cong \HH^0(X,\HHom_X(\cI_C,\OO_X)\ts\OO_L(-1))=0.   
  \end{equation}
  Moreover we have $\hom_X(\OO_X,\OO_L)=1$, 
  and \eqref{nosection-odd} for $F_{g+2}$ implies $\hom_X(F_{g+2},\OO_L)=1$ and $\Ext^1_X(F_{g+2},\OO_L)=0$.
  Then we deduce the vanishing:
  \begin{equation}
    \label{ciprovo5}
    \Ext^1_X(\cI_C,\OO_L(-1))=0.   
  \end{equation}
  Let us now  apply $\Hom_X(\cI_{C\cup L},-)$ to \eqref{ciprovo1}. We get:
  \[
  \Ext^1_X(\cI_{C\cup L},\OO_L(-1)) \to  \Ext^2_X(\cI_{C\cup L},\cI_{C\cup L}) \to \Ext^2_X(\cI_{C\cup L},\cI_C)
  \]
  We want to show that the middle term in the above sequence is zero,
  by showing that the remaining terms vanish.
  Applying $\Hom_X(-,\OO_L(-1))$ to \eqref{ciprovo1}, we get:
  \[
  \Ext^1_X(\cI_C,\OO_L(-1)) \to  \Ext^2_X(\cI_{C\cup L},\OO_L(-1)) \to \Ext^2_X(\OO_L,\OO_L).
  \]
  The leftmost term vanishes by \eqref{ciprovo5}, while the rightmost one does 
  by \cite[Remark 2.1]{brambilla-faenzi:genus-7}.
  It follows that $\Ext^2_X(\cI_{C\cup L},\OO_L(-1))=0$.
  Now, we apply $\Hom_X(-,\cI_C)$ to \eqref{ciprovo1}.
  We get:
  \[
  \Ext^2_X(\cI_C,\cI_C) \to  \Ext^2_X(\cI_{C\cup L},\cI_C) \to \Ext^3_X(\OO_L(-1),\cI_C).
  \]
  Note that $\Ext^3_X(\OO_L(-1),\cI_C) \cong
  \Hom_X(\cI_C,\OO_L(-2))^*=0$, where the vanishing follows from \eqref{ciprovo6}.
  On the other hand, by Lemma \ref{ragione} we can assume
  $\Ext^2_X(\cI_C,\cI_C)=0$ hence $\Ext^2_X(\cI_{C\cup L},\cI_C)=0$.  
  
  Summing up, we conclude that 
  $\Ext^2_X(\cI_{C\cup L},\cI_{C\cup L})=0$ and, by applying 
  Riemann-Roch we obtain 
  $\ext^1_X(\cI_{C\cup L},\cI_{C\cup L})=g+3$.
  By semicontinuity, we obtain the same vanishing for the
  curve $D$ as well. We have thus shown \eqref{ciprovo11}, and this
  finishes the proof.
\end{proof}

\section{Bundles with even first Chern class}  \label{sec:even}

We let again $X$ be any smooth non-hyperelliptic prime Fano threefold.
In this section, we study semistable sheaves $F$ with Chern classes
$c_1(F)=0$, $c_2(F)=4$, $c_3(F)=0$ on $X$,
and we prove the existence of case \eqref{4} of Madonna's list.
The main result of this part is the following.

\begin{thm} \label{thm:caso4}
Let $X$ be a smooth non-hyperelliptic prime Fano threefold.
Then there exists a rank $2$ ACM stable locally free sheaf
$\f$ with $c_1(\f)=0$, $c_2(\f)=4$.
The bundle $F$ lies in a generically smooth 
component of dimension $5$ of the
space $\Mo_X(2,0,4)$.
\end{thm}

We start with a review of some facts concerning conics contained in $X$.

\subsection{Conics and rank $2$ bundles with $c_2=2$}
Here we study rank $2$ sheaves on $X$ with $c_1=0$, $c_2=2$, and their
relation with the Hilbert scheme $\sH^0_2(X)$ of conics contained in $X$.
We rely on well-known properties of the Hilbert scheme $\sH^0_2(X)$,
see Section \ref{prime-fano}.

\begin{lem}
\label{lem:coniche}
Any Cohen-Macaulay curve $C \subset X$ of degree $2$ has $p_a(C)\leq0$.
Moreover if $C$ is non-reduced it must be a Gorenstein double structure
on a line $L$ defined by the exact sequence:
\begin{equation}\label{eq:manolache}
0 \to \cI_C \to \cI_L \to \OO_L(t)\to0,
\end{equation}
where $t \geq -1$ and we have $p_a(C)=-1-t$ and $\omega_C \cong \OO_C(-2-t)$.
\end{lem}

\begin{proof}
If $C$ is reduced, clearly it must be a conic (then $p_a(C)=0$) or the
union of two skew lines (then $p_a(C)=-1$).
So assume that $C$ is non-reduced, hence a double structure on a line
$L$. By \cite[Lemma 2]{manolache:cohen-macaulay-nilpotent} we have the
exact sequences:
\begin{align}
\nonumber & 0 \to \cI_C/\cI_L^2 \to \cI_L/\cI_L^2 \to \OO_L(t) \to 0, \\
\label{figlia-di-ferrand}
& 0 \to \OO_L(t) \to \OO_C \to \OO_L \to 0,
\end{align}
and $C$ is a Gorenstein structure given by Ferrand's doubling (see
\cite{manolache:multiple-on-smooth}, \cite{banica-forster}).
Recall that $\cI_L/\cI_L^2 \cong N^*_L$.
By \cite[Lemma 3.2]{iskovskih:II} we have either
$N_L^* \cong \OO_L\oplus \OO_L(1)$, or
$N_L^* \cong \OO_L (-1)\oplus \OO_L(2)$.
It follows that  $t \geq -1$ and we obtain \eqref{eq:manolache}.
We compute that $c_3(\cI_L)=-1$ and $c_3(\OO_L(t))=1+2t$, hence
$c_3(\cI_C)=-2-2t$, so $p_a(C)=-1-t$.

Dualizing \eqref{eq:manolache}, by the fundamental local isomorphism
we obtain the exact sequence:
\begin{equation}\label{eq:i-canonici}
0 \to \OO_L(-2) \to \omega_C \to \OO_L(-2-t) \to 0,
\end{equation}
which by functoriality is \eqref{figlia-di-ferrand} twisted by $\OO_X(-2-t)$.
This concludes the proof.
\end{proof}

\begin{corol}
All conics contained in $X$ are reduced if and only if $\sH^1_0(X)$ is smooth.
This takes place if $X$ is general.
\end{corol}

\begin{proof}
From \cite[Proposition 4.2.2]{fano-encyclo}, the Hilbert scheme
$\sH^1_0(X)$ is smooth if and only if we have $N_L\cong \OO_L\oplus
\OO_L(-1)$ for any line $L$ in $X$.
By the previous lemma this is equivalent to the fact that any conic contained in $X$ is reduced.
Recall that if $X$ is general, by \cite[Theorem 4.2.7]{fano-encyclo},
we have that $\sH^1_0(X)$ is smooth.
\end{proof}

Given a conic $D$, in view of Proposition \ref{hartshorneserre} (and by
Lemma \ref{lem:coniche}), there is a $\mu$-semistable vector bundle $\FD$
with $c_1(\FD)=0$, $c_1(\FD)=2$, which fits into:
\begin{equation} \label{eq:FD}
0 \to \OO_X \overset{\varphi}{\to} \FD \to \cI_D \to 0.
\end{equation}
One can easily prove the vanishing  $\Ext^2_X(\FD,\FD)=0$, since the normal bundle to $D$ is trivial
for generic $D$.

\begin{lem}\label{lem:ext2}
Let $F$ be a locally free sheaf on $X$, $C
\subset X$ a conic with normal bundle $N_C \cong \OO_C^2$, and $x$ a point of $C$.
Assume that $F \ts \OO_C \cong \OO_C ^2$ and that $\Ext^2_X(F,F)=0$.
Let $\sF$ be a sheaf fitting into an exact sequence of the form:
\begin{equation} \label{eq:FEC}
0 \to \sF \to F \to \OO_C \to 0.
\end{equation}

Then we have $\HH^0(C,\sF(-x))=0$ and $\Ext^2_X(\sF,\sF) = 0$.
\end{lem}

\begin{proof}
To prove the vanishing of $\HH^0(C,\sF(-x))$, we tensor \eqref{eq:FEC}
by $\OO_C$ and  we get the following exact sequence of sheaves on $C$:
\begin{equation} \label{eq:tor-conica}
0 \rr \TTor^X_1(\OO_C,\OO_C) \rr \sF \ts\OO_C \rr F \ts\OO_C  \rr \OO_C \rr 0.
\end{equation}
Recall that $\TTor_1^X(\OO_C,\OO_C)$ is isomorphic to $N^*_C \cong \OO_C^2$.
Now, twisting \eqref{eq:tor-conica} by $\OO_C(-x)$ and taking global
sections, we easily get $\HH^0(C,\sF(-x)) = 0$.
Now let us prove the vanishing of $\Ext^2_X(\sF,\sF)$.
Applying the functor $\Hom_X(-,\sF)$ to
\eqref{eq:FEC} we obtain the exact sequence:
\[
\Ext^2_X(F,\sF) \rr \Ext^2_X(\sF,\sF) \rr \Ext^3_X(\OO_C,\sF).
\]

We will prove that both the first and the last term
of the above sequence vanish.
Consider the former, and apply $\Hom_X(F,-)$ to
\eqref{eq:FEC}.
We get the exact sequence:
$$ \Ext^1_X(F,\OO_C) \rr \Ext^2_X(F,\sF) \rr \Ext^2_X(F,F).$$
By assumption we have $\Ext^2_X(F,F)=0$ and
$\Ext^1_X(F,\OO_C) \cong \HH^1(C,F)=0$. We obtain $\Ext^2_X(F,\sF)=0$.
To show the vanishing of the group $\Ext^3_X(\OO_C,\sF)$,
we apply Serre duality, obtaining:
\[
\Ext^3_X(\OO_C,\sF)^* \cong \Hom_X(\sF,\OO_C(-1)) \cong \HH^0(X,\HHom_X(\sF,\OO_C(-1))).
\]
To show that this group is zero, apply the functor $\HHom_X(-,\OO_C)$
to the sequence \eqref{eq:FEC} to get:
\[
0 \to \OO_C \to F^* \ts \OO_C \to \HHom_X(\sF,\OO_C) \to N_C \to 0,
\]
which implies $\HHom_X(\sF,\OO_C(-1)) \cong \OO_C(-1)^3$, and this sheaf
has no non-zero global sections.
\end{proof}

\begin{lem} \label{its-simple}
  Let $C$ and $D$ be smooth disjoint conics contained in $X$ with
  trivial normal bundle.
  Then a sheaf $\sF$ fitting into a
  nontrivial extension of the form:
  \begin{equation}  \label{eq:F_2JJ}
    0 \rr \cI_C \rr \sF  \rr \cI_{D}  \rr 0
  \end{equation}
  is simple.
\end{lem}
\begin{proof}
In order to prove the simplicity, apply $\Hom_X(\sF,-)$
to \eqref{eq:F_2JJ} and get
\[
\Hom_X(\sF,\cI_C)\to \Hom_X(\sF,\sF)\to \Hom_X(\sF,\cI_{D}).
\]

The first term vanishes, indeed applying $\Hom_X(-,\cI_C)$ to
\eqref{eq:F_2JJ} and since $C\cap D=\emptyset$ we get
\[
0\to \Hom_X(\sF,\cI_C)\to \Hom_X(\cI_C,\cI_C) \overset{\delta}{\to} \Ext^1_X(\cI_D,\cI_C).
\]

Clearly the map $\delta:\C \to \C$ is non-zero, hence $\Hom_X(\sF,\cI_C)=0$.
On the other hand, applying $\Hom_X(-,\cI_D)$ to
\eqref{eq:F_2JJ} we get:
\[
\Hom_X(\sF,\cI_D)\cong  \Hom_X(\cI_{D},\cI_{D})\cong \C,
\]
and we deduce $\hom_X(\sF,\sF) = 1$, i.e. the sheaf $\sF$ is simple.
\end{proof}

\subsection{ACM bundles of rank $2$ with $c_1=0$ and $c_2=4$}

This section is devoted to the proof of Theorem \ref{thm:caso4}.
The idea is to produce the required ACM bundle of rank $2$ as a
deformation of a simple sheaf obtained as
extension of the ideal sheaves of two sufficiently general conics.

\begin{step} \label{step1}
  Choose two smooth disjoint conics $C$ and $D$ in $X$ with trivial
  normal bundle.
\end{step}
It is well-know that there are two smooth conics $C$ and $D$ in $X$ with trivial normal
bundle, see Section \ref{prime-fano}.
Let us check that we can assume that $C$ and $D$ are disjoint.
Let $S$ be a hyperplane section surface containing $C$.
A general conic $D$ intersects $S$ at $2$ points. Since $X$ is covered
by conics, moving $D$ in
$\sH^0_2(X)$, these two points sweep out $S$.
Thus, a general conic $D$ meets $C$ at most
at a single point. This gives a rational map $\varphi:\sH^0_2(X)\to
C$.
Note that, for any point $x\in
C$, we have $\HH^0(C,N_C(-x))=0$. So there are only finitely many
conics contained in $X$ through $x$, for this space parametrizes the
deformations of $C$ which pass through $x$. Thus the general fibre of $\varphi$
is finite, which is a contradiction. 

\begin{step} \label{step2}
  Given the conics $C$ and $D$, define a simple sheaf $\sF$ with:
  \[c_1(\sF) = 0,  \quad  c_2(\sF) = 4, \quad   c_3(\sF) = 0.\]
\end{step}
Given the conic $D$, we have the bundle $\FD$ fitting in \eqref{eq:FD}.
Tensoring by $\OO_C$ this exact sequence,
we obtain $\FD_C \cong \OO_C^2$. Then we have $\hom_X(\FD,\OO_D)=2$,
and for any non-zero morphism $\FD_C\to \OO_C$ we denote by $\sigma$ the
surjective composition $\sigma: \FD \to \FD_C \to \OO_C$.
We can choose $\sigma$ such that the composition $\sigma \circ \varphi :
\OO_X \to \OO_C$ is non-zero, i.e. such that:
\begin{equation}
  \label{eq:condizione}
  \ker(\sigma) \not\supset \im(\varphi)\ts\OO_C.
\end{equation}
We denote by $\sF$ the kernel of $\sigma$ and we have the exact sequence:
\begin{eqnarray}
  &0 \rr \sF \rr \FD \overset{\sigma}{\rr} \OO_C \rr 0, \label{eq:F_2}
\end{eqnarray}
and patching this exact sequence together with \eqref{eq:FD}, we see that $\sF$ fits into \eqref{eq:F_2JJ}.
It is easy to compute the Chern classes of $\sF$, and to prove that $\sF$ is stable.
By \eqref{eq:F_2JJ} we get $\HH^k(X,\sF)= 0$ for all $k$.
More than that, since a smooth conic is projectively normal, again by
\eqref{eq:F_2JJ} we obtain:
\begin{equation} \label{eq:H1}
\HH^1_*(X,\sF) = 0.
\end{equation}
Note that the sheaf $\sF$ is strictly
semistable, and simple by Lemma \ref{its-simple}.
This concludes Step \ref{step2}.

\begin{step} \label{step3}
  Flatly deform the sheaf $\sF$ to a simple sheaf $G$ which does not fit into
  an exact sequence of the form \eqref{eq:F_2}.
\end{step}

Note that Lemma \ref{lem:ext2} gives $\Ext^2_X(\sF,\sF)=0$.
Hence, by \cite{artamkin} we know that there exists a universal
deformation of the simple sheaf $\sF$.
Since semistability is an open property, by a result of Maruyama,
we may assume that the deformation of $\sF$ is semistable. In other
words we can deform $\sF$ in the open subset $\Sigma$ of $\Spl_X$
given by simple semistable sheaves of rank $2$ and Chern classes 
$c_1=0$, $c_2=4$.
By Riemann-Roch and by the
simplicity of $\sF$ we get that $\ext^1(\sF,\sF)=5$.
This implies that $\Sigma$ is locally of dimension $5$ 
around the point $[\sF]$.

Now we want to prove that the set of sheaves in $\Sigma$ fitting into
an exact sequence of the form \eqref{eq:F_2} forms a subset of
codimension $1$ in $\Sigma$. 

  We have proved that a sheaf $\sF$ fitting into an exact sequence
  of the form \eqref{eq:F_2}, for some disjoint conics $C,D \subset
  X$, fits also into \eqref{eq:F_2JJ}. So, we need only prove that
  the set of sheaves fitting into \eqref{eq:F_2JJ}
  is a closed subset of dimension $4$ of $\Sigma$.
  Since $C$ and $D$ belong to the surface $\sH^0_2(X)$,
  it is enough to prove that there is in fact a unique (up to isomorphism) nontrivial such
  extension, i.e. $\ext^1_X(\cI_D,\cI_C)=1$.

  Note that
  $\Hom_X(\cI_D,\cI_C)=\Ext^3_X(\cI_D,\cI_C)=0$, hence by Riemann-Roch it suffices to prove
  $\Ext^2_X(\cI_D,\cI_C)=0$. Applying $\Hom_X(-,\cI_C)$ to the exact
  sequence defining $D$ in $X$, we are reduced to prove the vanishing
  of $\Ext^3_X(\OO_D,\cI_C)$. But this group is dual to:
  \[
  \Hom_X(\cI_C,\OO_D(-1)) \cong \HH^0(X,\HHom_X(\cI_C,\OO_D(-1))) \cong
  \HH^0(X,\OO_C(-1)) = 0.
  \]

So we have proved that $\ext^1_X(\cI_D,\cI_C)=1$ and that the set of
sheaves in $\Sigma$ fitting into an exact sequence of the form
\eqref{eq:F_2} has codimension $1$ in $\Sigma$. 
Then we can choose a deformation $G$ of $\sF$ in $\Sigma$ which does
not fit into \eqref{eq:F_2}, whereby concluding Step \ref{step3}.

Setting $E = G^{**}$, we write the double dual
sequence:
\begin{equation} \label{eq:FET}
0 \to G \to E \to T \to 0.
\end{equation}

\begin{step}
  Compute the Chern classes of $T$, and prove:
  \[
  c_1(T)=0, \qquad -c_2(T) \in \{1,2\}. 
  \]
\end{step}

By semicontinuity, we may assume $\hom_X(G,G)=1$, and
$\HH^1(X,G(-1))=0$. We may also assume that, for any given line $L$
contained in $X$, we have the vanishing $\Ext^1_X(\OO_L(t),G) = 0$ for
all $t \in \Z$.
Indeed, applying $\Hom_X(\OO_L(t),-)$ to \eqref{eq:F_2} we get:
\[
\Hom_X(\OO_L(t),\OO_C) \to \Ext^1_X(\OO_L(t),\sF) \to \Ext^1_X(\OO_L(t),\FD),
\]
and observe that the leftmost term vanishes as soon as $L$ is not contained in $C$
(but $C$ is irreducible), while the rightmost does for $\FD$ is
locally free.
Clearly $E$ is a semistable sheaf, so $\HH^1(X,G(-1))=0$ implies
$\HH^0(X,T(-1))=0$, hence $T$ must be a pure sheaf supported on a
Cohen-Macaulay curve $B \subset X$. Summing up, we have $c_1(T)=0$ and $c_2(T) < 0$.

Let us show $c_2(T) \ge -2$.
We have already proved that $\HH^0(X,T(-1))=0$  and this implies that
$\chi(T(t))=-\hh^1(X,T(t))$ for any negative integer $t$.
Recall that, by \cite[Remark 2.5.1]{hartshorne:stable-reflexive},
the reflexive sheaf $E$ satisfies $\HH^1(X,E(t))=0$ for all $t\ll 0$.
Thus, tensoring by $\OO_X(t)$ the exact sequence \eqref{eq:FET} and taking
cohomology, we obtain $\hh^1(X,T(t))\le \hh^2(X,G(t))$ for all $t\ll 0$.
Further, for any integer $t$, we can easily compute the following Chern classes:
\[c_1(T(t))=0, \quad  c_2(T(t))=c_2(T)=4-c_2(E), \quad
c_3(T(t))=c_3(E)-2tc_2(T), \]
hence by Riemann-Roch formula we have
\[
\chi(T(t))= - t c_2(T) + \frac{1}{2}(c_3(E)-c_2(T)).
\]
Since $G$ is a general deformation of the sheaf $\sF$, we also have,
by semicontinuity, $\hh^2(X,G(t))\le \hh^2(X,\sF(t))$.
On the other hand, by \eqref{eq:F_2} we have
$\hh^2(X,\sF(t))=\hh^1(X,\OO_C(t))=-\chi(\OO_C(t))=-2t-1$.
Summing up we have, for all $t\ll0$, the following inequality:
\[
- t c_2(T) + \frac{1}{2}(c_3(E)-c_2(T))\ge 2t+1,
\]
which implies $c_2(T)\ge -2$.

\begin{step}
  Prove that $B$ must be a smooth conic, and deduce that:
  \[T \cong \OO_B.\]
\end{step}

Assume the contrary, and note that either $T \cong \OO_{L_1}(a_1)$ for some line $L_1\subset X$
and some $a_1\in \Z$ (if $c_2(T)=-1$), or, if $c_2(T)=-2$,
in view of Lemma \ref{lem:coniche} there must be a second line $L_2 \subset X$ (possibly
coincident with $L_1$), and $a_2 \in \Z$ such that $T$ fits into:
\begin{equation}
  \label{eq:duerette}
  0 \to \OO_{L_1}(a_1) \to T \to \OO_{L_2}(a_2) \to 0.
\end{equation}
But we have seen that $\Ext^1_X(\OO_L(t),\sF) = 0$ for all $t \in \Z$,
and for any line $L \subset X$. By semicontinuity we get
$\Ext^1_X(\OO_L(t),G) = 0$, for all $t\in \Z$ and any line $L\subset X$.
In particular $\Ext^1_X(\OO_{L_i}(a_i),G) = 0$,
for $i=1,2$, so $\Ext^1_X(T,G) = 0$ and \eqref{eq:FET} should split,
which is absurd.

Therefore $T$ must be of the form $\OO_B(a\, x)$, for some integer
$a$, and for some point $x$ of a smooth conic $B\subset X$.
By \cite[Proposition 2.6]{hartshorne:stable-reflexive}, we have
$c_3(E) = c_3(T) = 2\,a \geq 0$, while $\HH^0(X,T(-1)) = 0$ implies
$a-2 < 0$.

We are left with the cases $a=0$ and $a=1$ and we want to exclude the latter.
We do this by proving that $\Ext^1_X(\OO_B(x),G)$ is zero for any
conic $B \subset X$. This fact can be checked by semicontinuity
applying $\Hom_X(\OO_B(x),-)$ to \eqref{eq:F_2}, obtaining:
\[
\Hom_X(\OO_B(x),\OO_C) \to \Ext^1_X(\OO_B(x),\sF) \to \Ext^1_X(\OO_B(x),\FD).
\]
The rightmost term in the above sequence vanishes because $\FD$ is
locally free. The leftmost term is isomorphic to
$\HH^0(X,\HHom(\OO_B(x),\OO_C))$, and the sheaf
$\HHom(\OO_B(x),\OO_C)$ is zero for $B \neq C$.
On the other hand, if $B = C$ we have
$\Hom_X(\OO_C(x),\OO_C) \cong \HH^0(C,\OO_C(-x)) = 0$.

\medskip

We have thus obtained $T \cong \OO_B$.
But then $G$ would fit into an exact sequence of the form
\eqref{eq:F_2}, a contradiction.
Summing up, we have proved that $T$ must be zero, so $G$ is isomorphic
to $E$, and thus locally free.
Since $\HH^0(X,G)=0$, the sheaf $G$ must be stable. By \eqref{eq:H1} and 
semicontinuity we can assume
$\HH_*^1(X,G(t))=0$, so by Serre duality we get that $G$ is ACM.
This concludes the proof of Theorem \ref{thm:caso4}.

\section{Applications} \label{apps}

We devote this final section to some applications of the existence
results for ACM bundles of rank $2$ to pfaffian hypersurfaces of
projective spaces and quadrics.

\subsection{Pfaffian cubics in a 4-dimensional quadric}

Here we show that the equation of a cubic hypersurface in a smooth
quadric $Q \subset \p^5$ can be written as the pfaffian of a skew-symmetric
$6 \times 6$ matrix of linear forms on the coordinate ring $R(Q)$.

\begin{thm} \label{cisiprova2}
  Let $X$ be a smooth prime Fano threefold of genus $4$ contained in a
  non-singular quadric hypersurface $Q \subset \p^5$.
  Then the equation $f$ of $X$ in the coordinate ring of $Q$ is the pfaffian of a skew-symmetric
  matrix $M$ representing a map:
  \begin{align}
    \label{pfaffiano6}
    & \psi : \OO_{Q}(-1)^6 \to \OO_{Q}^6.
  \end{align}
\end{thm}

Recall that we denote by $\cS_1$ and $\cS_2$ the two non-isomorphic
spinor bundles on $Q$, see Section \ref{sezioneACM}. 
We denote by $\iota$ the inclusion of $X$ in $Q$ and by $H_X$ the
hyperplane class of $X \subset \p^5$.

\begin{lem} \label{cisiprova1}
  Let $X$ be as above, let $F_i$ be the restriction of $\cS_i$ to $X$.
  Let $C$ be a conic contained in $X$. Then we have:
  \[
  \HH^k(X,F_i \ts \cI_{C,X})=0,
  \]
  for all $k = 0,\ldots,3$.
\end{lem}
\begin{proof}
  The inclusions $C\subset X \subset Q$ induce an exact sequence:
  \[
  0 \to \OO_Q(-3) \to \cI_{C,Q} \to \cI_{C,X} \to 0.
  \]
Recall that the bundles $\cS_i$
  are ACM, and by stability we have 
\begin{align}
  \label{ciriprovo1} & \HH^0(Q,\cS_i(-1)) = 0.
\end{align}
  Hence twisting the above sequence by $\cS_i$, we are reduced to show 
  $\HH^k(Q,\cS_i \ts \cI_{C,Q})=0$.
  Now, the conic $C$ is the intersection of the quadric $Q$ and of
  three hyperplanes of $\p^5$. Thus we have an exact sequence:
  \[
  0 \to \OO_Q(-3) \to \OO_Q(-2)^3 \to \OO_Q(-1)^3 \to \cI_{C,Q} \to 0.
  \]
  Twisting the above exact sequence by $\cS_i$ and taking cohomology,
  we obtain the result, using \eqref{ciriprovo1} and the fact that the bundles $\cS_i$ are
  ACM.
\end{proof}

\begin{lem} \label{cisiriprova3}
  Let $X$ be as above, $E$ be a stable locally free sheaf in $\Mo_X(2,0,4)$.
  Then we have:
  \begin{align}
    \label{ciriprovo2} & \Ext_Q^1(\iota_*(E(1)),\cS_i(a)) = 0, && \mbox{for
      $a\leq -3$.}
    \intertext{If the sheaf $E$ is general in the component of
      $\Mo_X(2,0,4)$ provided by Theorem \ref{thm:caso4}, then we also have:}
    \label{ciriprovo3} & \Ext_Q^1(\iota_*(E(1)),\cS_i(-2)) = 0.
  \end{align}
\end{lem}

\begin{proof}
We first prove \eqref{ciriprovo2}. Since
$\HHom_Q(\iota_*(E(1)),\cS_i(a))=0$, the local-to-global spectral
sequence provides an isomorphism: 
\[
\Ext_Q^1(\iota_*(E(1)),\cS_i(a)) \cong
\HH^0\left(Q,\EExt_Q^1(\iota_*(E(1)),\OO_Q) \ts \cS_i(a)\right).
\]
By Grothendieck duality, we have:
\[
\EExt_Q^1(\iota_*(E(1)),\OO_Q) \cong \iota_*(E^*(-1)) \ts \OO_Q(3)
\cong \iota_*(E(2)),
\]
where the second isomorphism holds since $E$ is locally free.
Therefore we are reduced to show:
\begin{equation}
  \label{ciriprovo4}
  \HH^0(Q,\iota_*(E(2)) \ts \cS_i(a)) \cong \HH^0(X,E \ts F_i(2+a)) = 0,  
\end{equation}
for $a\leq -3$. Note that the sheaf $E \ts F_i(2+a)$ is semistable of
slope $5/2+a$, hence it has no non-zero global sections if $a\leq
-3$. Hence \eqref{ciriprovo2} is proved.

In order to prove \eqref{ciriprovo3}, we let
$E$ be general in the component provided by Theorem \ref{thm:caso4}, and we
show that \eqref{ciriprovo4} holds for $a=-2$.
In particular, we assume that $E$ is a deformation of a simple sheaf $\sF$
given as the middle term of an extension of the form \eqref{eq:F_2JJ}
of two ideal sheaves $\cI_D$, $\cI_C$ of two conics $C$, $D$ contained
in $X$.
By semicontinuity, it will thus suffice to show that:
\[
\HH^0(X,\sF \ts F_i) = 0,
\]
for $i=1,2$.
In turn, since $\sF$ is an extension of ideal sheaves of conics,
it will be enough to prove $\HH^0(X,\cI_C\ts F_i) = 0$,
where $C$ is a conic contained in $X$.
But we have shown this in Lemma \ref{cisiprova1}.
\end{proof}

\begin{proof}[Proof of Theorem \ref{cisiprova2}]
  Let $E$ be a stable ACM bundle of rank $2$ in the component of
  $\Mo_X(2,0,4)$ provided by Theorem \ref{thm:caso4}. By the previous lemma we may 
  assume that $E$ satisfies the cohomology vanishing conditions \eqref{ciriprovo2} and \eqref{ciriprovo3}.
  
  We consider a sheafified minimal graded free resolution of
  $\iota_*(E(1))$.
  In particular, we have a bundle $\sP$ on $Q$ of the form
  \begin{equation}
    \label{PP}
    \sP = \bigoplus_{i=1}^s \OO_Q(b_i), \qquad \mbox{with $b_1\geq \cdots \geq b_s$},    
  \end{equation}
  equipped with a projection $\pi : \sP \to
  \iota_*(E(1))$ such that $\pi$ induces a surjective map:
  \begin{equation}
    \label{ciriprovo5}
    \HH_*^0(Q,\sP) \to \HH_*^0(Q,\iota_*(E(1))).    
  \end{equation}
  Let $\sK$ be the kernel of $\pi$.
  It is clear that, since \eqref{ciriprovo5} is surjective, we have:
  \[
  \HH^1_*(Q,\sK)=0.
  \]
  Moreover, since $E$ is ACM on $X$ and $\sP$ is ACM on $Q$, it
  easily follows that the sheaf $\sK$ is ACM on $Q$.
  By a well-known theorem of Kn\"orrer (see \cite{knorrer:aCM}), this implies that $\sK$ splits as a direct
  sum:
  \begin{equation}
    \label{ciriprovo6}
    \sK \cong \bigoplus_{j=1}^{t} \cS_{i_j}(c_j)\oplus
    \bigoplus_{h=1}^{u} \OO_Q(a_h), \qquad \mbox{with $a_1\geq \cdots \geq a_u$},
  \end{equation}
  where $i_j\in\{1,2\}$.
  Note that, since $\HH^0(X,E)=0$, we have $b_i \leq 0$ for all $i$.
  Therefore we also have $a_i \leq -1$ (by the minimality of the
  resolution) and $c_j \leq -1$
  (since $\Hom_{Q}(\cS_{i}(c),\OO_{Q})=0$ for $c \geq 0$ by \eqref{ciriprovo1}).

  Let us now use Lemma \ref{cisiriprova3}. The vanishing results
  \eqref{ciriprovo2} and \eqref{ciriprovo3} imply that no $c_j \leq - 2$
  occurs in the expression \eqref{ciriprovo6}.
  For in that case, it would easily follow that $\sP$ contains $\cS_{i_j}(c_{j})$
  as a direct summand, which is not the case.
  The remaining possibility is excluded by the following:
  \begin{claim} \label{cisiprova3} 
    We have $t=0$ and $s=6$.
    Moreover, for all $i=1,\ldots,6$, we have $b_i=0$, $a_i=-1$.
  \end{claim}
  Once the above claim is proved, the proof of Theorem
  \ref{cisiprova2} will be finished. Indeed, the sheaf $\sK$ is a direct sum
  of line bundles. Therefore, the argument of \cite[Theorem
  B]{beauville:determinantal}
  applies to our setup, and the matrix $M$ representing the morphism
  $\psi : \sK \to \sP$ can be chosen skew-symmetric,
  with pfaffian equal to the equation defining $X \subset Q$.
\end{proof}

\begin{proof}[Proof of Claim \ref{cisiprova3}]
We have the exact sequence:
\[
0 \to \sK \xr{\psi} \sP \to \iota_*(E(1)) \to 0.
\]
Recall that in view of the above analysis, $c_{j}$ can only be $-1$, 
while $a_i \leq -1$ and $b_i \leq 0$ for all $i$.
One easily computes $\hh^{0}(X,E(1))=6$, so that $b_{i}=0$ for
$i=1,\ldots , 6$ and $b_{i} \leq -1$ for $i \geq 7$.
Then, recalling that $c_{j}=-1$ for all $j$ and noting that $\cS_{i_{j}}(-1)$ must be mapped by
the injective map $\psi : \sK \to \sP$ to $\OO_{Q}^{6}$, we deduce
that $t \leq 3$. The two equations $\rk(\iota_*(E(1)))=0$ and
$c_{1}(\iota_*(E(1))) = 6H_{Q}$, imply respectively:
\begin{align}
\nonumber & 2t + u = s,  \\
\label{fammiprovare2} & \sum_{j=1}^{s} b_{j} - \sum_{i=1}^{s-2t} a_{i} + t = 6.
\end{align}
To prove our statement, we adapt an argument of Bohnhorst-Spindler,
see \cite{bohnhorst-spindler}.
Namely, we write
$\psi=(\psi'_{r,1},\ldots,\psi'_{r,t},\psi''_{r,1},\ldots,\psi''_{r,u})_{1\leq
r\leq s}$ and we note that $\psi'_{r,j}=0$ for any $r\geq 7$ and $1\leq j\leq t$.
Now for each $\ell \leq s-2t=u$, we let $r_{\ell}$ be the maximum
integer $r$ such that:
\[
(\psi'_{r,1},\ldots,\psi'_{r,t},\psi''_{r,1},\ldots,\psi''_{r,\ell})\neq 0.
\]
Since the map $\psi$ is injective, this easily implies:
\[
2t +\ell \leq r_{\ell}.
\]
So, there must be $j\leq t$ such that $\psi'_{r_{\ell},j} \neq 0$, or $j \leq \ell$ such that $\psi''_{r_{\ell},j} \neq 0$.
In the first case we have $r_\ell \leq 6$ so $2t + \ell \leq 6$ hence
$b_{2t+\ell} = 0$.
In the second case, by the minimality of the
resolution map $\psi$, we get
$b_{r_{\ell}} - a_{j} \geq 1$.
We deduce, for each $\ell \leq s-2t$, the inequality
$
  b_{2t+\ell} - a_{\ell} \geq b_{r_{\ell}} - a_{\ell} \geq b_{r_{\ell}} - a_{j} \geq 1.
$
In both cases we have:
\begin{equation}
  \label{fammiprovare3}
  b_{2t+\ell}-a_\ell \geq 1.
\end{equation}
From the equation \eqref{fammiprovare2} we get:
\[
t + \sum_{\ell=1}^{s-2t} (b_{2t+\ell} - a_{\ell}) \leq 6,
\]
and by \eqref{fammiprovare3} we obtain:
\[
s-t = t + s - 2t \leq t + \sum_{\ell=1}^{s-2t} (b_{2t+\ell} - a_{\ell}) = 6.
\]
One can now easily compute
$c_{2}(\iota_{*}(E(1)))=21(\Lambda_{1}+\Lambda_{2})$, which implies
that $t$ must be even, for otherwise we would have
$c_{2}(\sK)=\alpha_{1}\Lambda_{1}+\alpha_{2}\Lambda_{2}$ with
$\alpha_{1} \neq \alpha_{2}$, which easily leads to a contradiction.
Since $t \leq 3$, $s\geq 6$, and $b_{j} \leq -1$ for $j\geq 7$ we are left with the cases:
\begin{small}
\begin{align}
  &  t=0, \, s = 6, && \Longrightarrow && \underline{b}=(0,0,0,0,0,0),
  \,\, -\underline{a}=(1,1,1,1,1,1), \\
  &  t=2, \, s = 6, && \mbox{and} && \underline{b}=(0,0,0,0,0,0),
  \,\,  -\underline{a}=(2,2), \\
  &  t=2, \, s = 6, && \mbox{and}  && \underline{b}=(0,0,0,0,0,0),
  \,\,  -\underline{a}=(1,3), \\
  &  t=2, \, s = 7, && \mbox{and}  && \underline{b}=(0,0,0,0,0,0,b_{7}),
  \,\,  -\underline{a}=(1,1,2-b_{7}), \\
  &  t=2, \, s = 7, && \mbox{and}  && \underline{b}=(0,0,0,0,0,0,b_{7}),
  \,\,  -\underline{a}=(1,2,1-b_{7}), \\
  &  t=2, \, s = 8, && \Longrightarrow && \underline{b}=(0,0,0,0,0,0,b_{7},b_{8}),
  \,\, -\underline{a}=(1,1,1-b_{7},1-b_{8}),
\end{align}
\end{small}
where $\underline{a}$ and $\underline{b}$ denote the vectors of
$\Z^{s-2t}$ and $\Z^{s}$ representing the sequence of
$a_{i}$'s and $b_{j}$'s, and we have $-1 \geq b_{7} \geq b_{8}$.
It is now an easy exercise to check that, in all the above cases
except the first one, the difference of the Hilbert polynomials of
$\sP$ and $\sK$ does not equal $p(\iota_{*}(E(1)),t)= 2t^{3}+ 9t^{2}
+13t +6$. This is a contradiction, and leaves the
desired case as the only possibility.
\end{proof}

\subsection{Pfaffian quartic threefolds in $\p^4$}
A theorem of Iliev-Markushevich, \cite{iliev-markushevich:quartic}
asserts that a general quartic threefold $X$ in $\p^4$ is a {\it linear pfaffian}, namely its equations $f$
is the pfaffian of an $8 \times 8$ skew-symmetric matrix of linear
forms on $\p^4$. 
Similarly, a result of Madonna, \cite{madonna:quartic}, says that $X$ is a {\it quadratic pfaffian}, 
 that is, $f$ can be written as the pfaffian of a $4 \times 4$ skew-symmetric matrix of quadratic forms.
Their proofs are carried out with the aid of the computer algebra package {\tt Macaulay2}.

Here we prove a result in the same spirit, as an application of our existence
theorems. We show that any ordinary quartic threefold in $\p^4$ is a linear
pfaffian, and that any smooth quartic threefold in $\p^4$ is a quadratic
pfaffian, with at most two more rows and columns of linear forms.

\begin{thm} \label{ultimaprova5}
  Let $X$ be a smooth quartic threefold in $\p^4$ defined by an
  equation $f$.
  \begin{enumerate}[i)]
  \item \label{quadratica} Then there is a skew-symmetric matrix $M$ representing a map
    of one of the two forms:
    \begin{equation}
      \label{reconduit}
      \OO_{\p^4}(-2)^4 \to \OO_{\p^4}^4, \quad  \mbox{or:} \quad
      \OO_{\p^4}(-2)^4 \oplus \OO_{\p^4}(-1)^2 \to \OO_{\p^4}(-1)^2
      \oplus \OO_{\p^4}^4,
    \end{equation}
    such that $\Pf(M)=f$;
  \item \label{lineare} if $X$ is ordinary, there 
    is a skew-symmetric matrix $N$ representing a map:
    \[
    \OO_{\p^4}(-1)^8 \to \OO_{\p^4}^8,
    \]
    such that $\Pf(N)=f$.
  \end{enumerate}
\end{thm}

\begin{proof}
  We work as in Theorem \ref{cisiprova2}.
  Let $E$ be an ACM bundle on $X$, and consider the sheafified minimal graded free resolution of
  $\iota_*(E(1))$. We have a bundle $\sP$ on $\p^4$ of the form \eqref{PP}
  and a projection $\pi : \sP \to
  \iota_*(E(1))$ such that $\pi$ is surjective on global sections for
  each twist. The kernel $\sK$ of this projection is ACM on $\p^{4}$
  so we have:
  \[
  \sK \cong \bigoplus_{h=1}^{s} \OO_{\p^4}(a_h), \qquad \mbox{with $a_1\geq \cdots \geq a_s$}.
  \]

  The matrix representing the map $\sK \to \sP$ can be chosen skew-symmetric by
  \cite[Theorem B]{beauville:determinantal}, and its pfaffian is $f$.
  In particular the integer $s$ must be even.
  Assuming $\HH^0(X,E)=0$, we have $b_i \leq 0$ for all $i$,
  so by the minimality of the resolution $a_i \leq -1$.
  We further have:
  \begin{align}
    \label{ultimaprova2}
    & \sum_{\ell=1}^{s} \left( b_{\ell} - a_{\ell} \right) = 8,    
    \intertext{  and, by the argument of Bohnhorst-Spindler, we can assume:}
    \label{ultimaprova3} & b_{\ell} - a_{\ell} \geq 1, \qquad \mbox{for all $1\leq \ell \leq s$.}
  \end{align}

  Now, to prove \eqref{lineare} we choose as $E$ a general bundle with
  $c_{1}(E)=1$, $c_{2}(E)=6$ given by Theorem \ref{topolino}.
  It is straightforward to compute $\HH^{0}(X,E)=0$, and
  $\hh^{0}(X,E(1))=8$. Therefore we have $s = 8$ and $b_{i}=0$ for
  all $i$. Thus by \eqref{ultimaprova2} and \eqref{ultimaprova3} we
  have $a_{i}=-1$ for all $i$ and we are done.

  Let us now show \eqref{quadratica}. This time we pick a general bundle $E$
  with $c_{1}(E)=0$, $c_{2}(E)=4$, provided by Theorem \ref{thm:caso4}.
  One computes $\HH^{0}(X,E)=0$, and
  $\hh^{0}(X,E(1))=4$, so that $b_{i}=0$ for $i=1,2,3,4$ and $b_{i}
  \leq -1$ for $i\geq 5$, while $s \in \{4,6,8 \}$.
  The Hilbert polynomial $p(\iota_{*}(E(1)),t)$ reads:
  \begin{equation}
    \label{hilbpoly}
    \frac{4}{3} t^{3} + 6 t^{2} + \frac{26}{3} t + 4.    
  \end{equation}

  To finish the proof, we divide it into different cases according to
  the value of $s$: we want to show that $E$ can be chosen so that
  $s=4$ with $a_{i}=-2$ for all $i$ or $s=6$ and $a_i=b_i=-1$ for $i=5,6$.
  \begin{case}[$s=8$]
    In this case, in view of \eqref{ultimaprova2} and
    \eqref{ultimaprova3} we must have $b_{i} - a_{i} = 1$ for
    all $i$, so $a_{i}=-1$ for $i=1,2,3,4$ and recall that $b_{i}
    \leq -1$.
    Looking at the quadratic term of the Hilbert polynomial, one sees
    that \eqref{hilbpoly} forces $b_{i}=-1$ for all $i=5,6,7,8$.
    Therefore the Pfaffian of the matrix $N$ is the square of the
    determinant of a $4\times 4$ matrix of linear forms, which is
    impossible since $f$ is not a square.
  \end{case}

  \begin{case}[$s=6$]
    In this case we have $b_{6} \leq b_{5} \leq -1$, and we have the possibilities:
     \begin{align*}
      -\underline{a} & = (1,1,2,2,1-b_{5},1-b_{6}), \\
      -\underline{a} & = (1,1,1,2,2-b_{5},1-b_{6}), \\
      -\underline{a} & = (1,1,1,2,1-b_{5},2-b_{6}), \\
      -\underline{a} & = (1,1,1,1,2-b_{5},2-b_{6}).
    \end{align*}
 
    Looking again at the quadratic term of the Hilbert polynomial, it
    is easy to see that the only case left by \eqref{hilbpoly} is the
    first one, with $b_{7}=b_8=-1$. This gives rise to the second
    alternative in \eqref{reconduit}.
  \end{case}

  \begin{case}[$s=4$]
    If $s=4$, there are finitely many choices for the $a_{i}$'s
    according to \eqref{ultimaprova2} and \eqref{ultimaprova3}.
    These are:
    \[
      -\underline{a} \in \{(1,1,1,5),(1,1,2,4), (1,1,3,3), (1,2,2,3), (2,2,2,2).\}
    \]
    A straightforward computation shows that only in the last case
    the Hilbert polynomial agrees with \eqref{hilbpoly}.
    Since this case corresponds to the first alternative in
    \eqref{reconduit}, this finishes the proof. 
  \end{case}
\end{proof}


\vspace{-0,4cm}

\def\cprime{$'$} \def\cprime{$'$} \def\cprime{$'$} \def\cprime{$'$}
  \def\cprime{$'$} \def\cprime{$'$}
\providecommand{\bysame}{\leavevmode\hbox to3em{\hrulefill}\thinspace}
\providecommand{\MR}{\relax\ifhmode\unskip\space\fi MR }
\providecommand{\MRhref}[2]{%
  \href{http://www.ams.org/mathscinet-getitem?mr=#1}{#2}
}
\providecommand{\href}[2]{#2}

\end{document}